\documentclass[9pt,reqno]{amsart}  

\usepackage[english]{babel}
\usepackage[p,osf]{cochineal}
\usepackage[scale=.95,type1]{cabin}
\usepackage[zerostyle=c,scaled=.94]{newtxtt}
\usepackage[T1]{fontenc} 
\usepackage[utf8]{inputenc} 
\usepackage{amsthm}
\usepackage{amsmath}
\usepackage{amssymb}
\usepackage{amsfonts}
\usepackage{amsaddr}
\usepackage{xspace}
\usepackage{enumitem} 
\usepackage{csquotes}
\usepackage{xcolor}
\usepackage[colorlinks]{hyperref}
\hypersetup{colorlinks, breaklinks, citecolor=teal,filecolor=black}
\usepackage{graphicx}
\usepackage{mathtools}
\mathtoolsset{centercolon}
\usepackage{bm}
\usepackage{pgfkeys}
\usepackage{pgf,interval}
\usepackage{tikz,pgf}
\intervalconfig{soft open fences}

\DeclareMathOperator{\Tr}{Tr}

\DeclareMathOperator{\E}{\mathbf{E}}

\DeclareMathOperator{\Av}{Av}
\DeclareMathOperator{\Cyc}{Cyc}
\DeclareMathOperator{\Gr}{Gr}

\newcommand{\ov}{\overline}

\newcommand{\ii}{\mathrm{i}}
\newcommand{\cred}[1]{{\color{red} #1}}

\newcommand{\F}{\mathbf{F}}

\ifdefined\C
\renewcommand{\C}{\mathbf{C}}
\else
\newcommand{\C}{\mathbf{C}}
\fi

\newcommand{\un}{\underline}
\newcommand{\vx}{\bm{x}}

\newcommand{\bu}{\bm{u}}
\newcommand{\bq}{\bm{q}}

\newcommand{\wt}{\widetilde}
\newcommand{\wh}{\widehat}

\newcommand{\R}{\mathbf{R}}
\newcommand{\N}{\mathbf{N}}
\newcommand{\Z}{\mathbf{Z}}

\newcommand{\cG}{\mathcal{G}}
\newcommand{\cO}{\mathcal{O}}
\newcommand{\co}{{\scriptstyle\mathcal{O}}}

\newcommand{\dif}{\operatorname{d}\!{}}
\newcommand\restr[3]{{%
  \left.\kern-\nulldelimiterspace %
  #1 %
  \vphantom{\big|} %
  \right|_{#2}^{#3} %
  }}
\DeclarePairedDelimiter{\braket}{\langle}{\rangle}%
\DeclarePairedDelimiter{\abs}{\lvert}{\rvert}%
\DeclarePairedDelimiter{\norm}{\lVert}{\rVert}%
\providecommand\given{}
\newcommand\SetSymbol[1][]{\nonscript\:#1\vert\allowbreak\nonscript\:\mathopen{}}
\DeclarePairedDelimiterX{\tuple}[1](){\renewcommand\given{\SetSymbol[\delimsize]}#1}
\DeclarePairedDelimiterX{\set}[1]\{\}{\renewcommand\given{\SetSymbol[\delimsize]}#1}
\DeclarePairedDelimiterXPP{\landauO}[1]{\cO}(){}{#1}
\DeclarePairedDelimiterXPP{\landauo}[1]{\co}(){}{#1}
\DeclarePairedDelimiterXPP{\landauOprec}[1]{\cO_\prec}(){}{#1}
\DeclarePairedDelimiterXPP{\landauOstd}[1]{\cO_\prec^2}(){}{#1}
\DeclarePairedDelimiterXPP{\landauOE}[1]{\cO_\prec^1}(){}{#1}
\DeclarePairedDelimiterXPP{\landauOd}[1]{\cO_\mathrm{d}}(){}{#1}

\usepackage[giveninits=true,url=false,doi=false,isbn=false,eprint=true,datamodel=mrnumber,sorting=nty,maxcitenames=4,maxbibnames=99,backref=false,block=space,backend=biber,bibstyle=phys]{biblatex} 
\AtEveryBibitem{\clearfield{month}}
\AtEveryCitekey{\clearfield{month}} 
\renewbibmacro{in:}{}
\DeclareFieldFormat[article]{title}{\emph{#1}} 
\DeclareFieldFormat{mrnumber}{\ifhyperref{\href{http://www.ams.org/mathscinet-getitem?mr=#1}{\nolinkurl{MR#1}}}{\nolinkurl{#1}}}
\DeclareFieldFormat{pmid}{\ifhyperref{\href{https://www.ncbi.nlm.nih.gov/pubmed/#1}{\nolinkurl{PMID#1}}}{\nolinkurl{#1}}}
\DeclareFieldFormat{eprint}{\ifhyperref{\href{https://arxiv.org/abs/#1}{\nolinkurl{arXiv:#1}}}{\nolinkurl{#1}}}
\renewbibmacro*{doi+eprint+url}{%
  \iftoggle{bbx:doi}{\printfield{doi}}{}
  \newunit\newblock%
  \printfield{mrnumber}%
  \newunit\newblock%
  \printfield{pmid}%
  \newunit\newblock%
  \printfield{eprint}%
  \iftoggle{bbx:url}{\usebibmacro{url+urldate}}{}}
\usetikzlibrary{calc}
\usetikzlibrary{arrows}
\bibliography{bibclt}

\numberwithin{equation}{section} 
\newtheorem{theorem}{Theorem}[section]
\newtheorem{assumption}[theorem]{Assumption}
\newtheorem{lemma}[theorem]{Lemma}
\newtheorem{proposition}[theorem]{Proposition}

\newtheorem{remark}[theorem]{Remark}

\date{\today}
\author{Giorgio Cipolloni \and L\'aszl\'o Erd\H{o}s}
\address{IST Austria, Am Campus 1, 3400 Klosterneuburg, Austria}
\author{Dominik Schr\"oder\(^{\dagger}\)}  
\address{Institute for Theoretical Studies, ETH Zurich, Clausiusstr.\ 47, 8092 Zurich, Switzerland}
\email{giorgio.cipolloni@ist.ac.at} 
\email{lerdos@ist.ac.at}
\email{dschroeder@ethz.ch}
\thanks{\(^{\dagger}\)Supported by Dr.\ Max R\"ossler, the Walter Haefner Foundation and the ETH Z\"urich Foundation}
\subjclass[2010]{60B20, 15B52} 
\keywords{Local Law, Dyson Brownian motion, Stochastic eigenstate equation, Eigenvector moment flow}
\title{Normal fluctuation in quantum ergodicity  for Wigner matrices}
\date{\today}

\begin{document} 

\begin{abstract} 
We consider the quadratic form of a general  deterministic matrix 
on the eigenvectors of an \(N\times N\) Wigner matrix and prove that it has Gaussian fluctuation
for each bulk eigenvector in the large \(N\) limit. The proof is a combination of the
energy method for the Dyson Brownian motion inspired by~\cite{2005.08425}
and our recent multi-resolvent local laws~\cite{2012.13215}.
\end{abstract}

\thispagestyle{empty}

\maketitle

\section{Introduction}
Quantum Unique Ergodicity (QUE) in a disordered or chaotic quantum system asserts
that the eigenvectors of the Hamilton operator tend to become uniformly distributed
in the phase space, see~\cite{MR0402834, MR818831, MR916129, MR1266075, MR1810753, MR3961083}
for the seminal  results and~\cite{2012.13215} for more recent references.
 We study a particularly strong form of this phenomenon 
for Wigner random matrices, the simplest prototype of a fully chaotic Hamiltonian.
These are  \(N\times N\)  random Hermitian   matrices \(W=W^*\)
with centred, independent, identically distributed \emph{(i.i.d.)} entries up to the symmetry constraint \(w_{ab} = \ov{w_{ba}}\).
Let \(\set{\bu_i}_{i=1}^N\) be an orthonormal eigenbasis of \(W\) corresponding to the eigenvalues 
\({\bm \lambda} = (\lambda_i)_{i=1}^N\) listed in increasing order.
Recently we showed~\cite{2012.13215} that for any  deterministic matrix \(A\)  with \(\|A\|\le1\),
the eigenvector overlaps \(\braket{\bu_i, A \bu_i}\)  
converge to  \(\braket{A}:=\frac{1}{N}\Tr A\), the normalized trace of \(A\), in the large \(N\) limit. More generally,
we proved that 
\begin{equation}\label{eq:eth}
  \max_{i,j}\abs[\Big]{\braket{ \bu_i, A \bu_j}-\braket{A}\delta_{ij}} 
  \lesssim \frac{N^\epsilon}{\sqrt{N}} 
\end{equation}  
holds with very high probability. We note that the bound~\eqref{eq:eth} is optimal for high-rank deterministic matrices \(A\) and is coined as the \emph{Eigenstate Thermalization Hypothesis} by 
Deutsch~\cite{9905246} and Srednicki~\cite{9962049}, see also~\cite[Eq.~(20)]{Dalessio2015}. 

The main result of the current paper, Theorem~\ref{theo:flucque},
 asserts that \(\braket{\bu_i, A \bu_i}\) has a  Gaussian fluctuation
on scale \(N^{-1/2}\), 
more precisely 
\begin{equation}\label{eq:flu}
   \sqrt{N} \Big[ \braket{\bu_i, A \bu_i}  - \braket{A}\Big]
\end{equation} 
converges to a normal distribution for any Hermitian observables \(A=A^*\) of high rank and
for any eigenvectors \({\bm u}_i\) whose eigenvalue  belongs to  the bulk of the spectrum.

For Gaussian ensembles and $A$ being a projection onto macroscopically many coordinates~\eqref{eq:flu} can be proven by using the special invariance property of the eigenvectors (see~\cite[Theorem 2.4]{MR3534074}). Our result concerns general Wigner matrices and it has two main features: it concerns individual eigenvectors \emph{and} it is valid for general high  rank observables.
We now explain related previous results which all addressed only one of these features.
First,  Gaussianity of~\eqref{eq:flu} after a small averaging 
in the index \(i\) has recently been established in~\cite[Theorem 2.3]{2012.13218} using resolvent methods. Second, fluctuations involving individual eigenvectors in  the bulk spectrum for general Wigner
matrices can only be accessed by the
Dyson Brownian motion  approach which has only been developed for finite rank observables~\cite{MR3606475, 2005.08425, MR4164858}.
We now explain the background of these concepts.

\subsection{Dyson Brownian motion for eigenvectors}
For the simplest rank one case,
\(A=|{\bm q}\rangle \langle {\bm q}|\) with some deterministic unit vector  \(\bq\), 
 Bourgade and Yau~\cite{MR3606475} showed  that the squared normalised overlaps
\(N |\langle \bu_i, \bq\rangle |^2\) converge in distribution  to the square of a standard Gaussian variable as \(N\to\infty\) (see also~\cite{MR3034787, MR2930379} for
the same result without DBM but under four moment matching condition in the bulk). Similar results have been obtained for deformed Wigner matrices~\cite{MR4164858}, for sparse matrices~\cite{MR3690289}, and for L\'evy matrices~\cite{MR4260468}. Note that both the scaling and the limit distribution for the rank one case are different from~\eqref{eq:flu}. The basic intuition
 is that the coordinates of \(\bu_i\) are roughly independent, thus the sum in \(\langle \bu_i, \bq\rangle=
 \sum_a \bu_i(a) \bq(a)\) obeys a central limit theorem (CLT) on scale \(N^{-1/2}\).
 In fact,~\cite{MR3606475} also considers the joint distribution of
 finitely many eigenvectors tested against one fixed vector \(\bq\)
and  the joint distribution 
of a single eigenvector with finitely many test vectors \(\bq_1, \bq_2, \ldots \bq_K\) for any fixed \(K\), independent of
 \(N\).
Very recently, Marcinek and Yau~\cite{2005.08425} have established that the overlaps of
finitely many eigenvectors \emph{and} finitely many orthogonal test vectors are also asymptotically independent (squared) normal. Their method is very general and also applies to a large class of other random matrix ensembles, such as sparse or L\'evy matrices.

The fundamental method behind all  results involving individual eigenvectors for general Wigner matrices
 is the Dyson Brownian motion (DBM) for
eigenvectors, also called the \emph{stochastic eigenstate equation} 
 generated by 
a simple matrix Brownian motion for \(W\), introduced by Bourgade and Yau in~\cite{MR3606475}. We briefly summarize the key steps in~\cite{MR3606475}
in order to highlight the new ideas we needed to prove the Gaussianity of~\eqref{eq:flu}. 

 For each fixed \(n\),  the evolution of the joint \(n\)-th order moments of the 
overlaps \(N |\langle \bu_i, \bq\rangle |^2\) for different \(i\)'s and fixed \(\bq\)
is  described by 
a system of parabolic evolution equations, called the \emph{eigenvector moment flow}. %
Interpreting 
each such overlap as a particle sitting at location \(i\) in the  discrete one dimensional
index space \([N] = \{ 1,2, \ldots , N\}\), the moment flow naturally corresponds to 
a Markovian jump  process of \(n\) particles. It turns out that the rate of a jump from site \(i\) to
site \(j\) is proportional with \(N^{-1}(\lambda_i-\lambda_j)^{-2}\).
Different \(\bq\)'s can be incorporated by appropriately
assigning \emph{colours} to the particles.
 By the fast local equilibration property of the DBM
the moments of \(N |\langle \bu_i, \bq\rangle |^2\) quickly become essentially independent of the index \(i\)
at least for  indices corresponding to nearby eigenvalues \(\lambda_i\),
hence they can be computed by locally averaging over \(i\). For example, in the 
simplest \(n=1\) case we have 
\begin{equation}\label{2ndmoment}
    f_i:= \E \big[ N |\langle \bu_i, \bq\rangle |^2 \big| {\bm \lambda}\big]  \approx  
  f_{i'}=    \E \big[ N |\langle \bu_{i'}, \bq\rangle |^2 \big| {\bm \lambda}\big],
     \quad  |i-i'|\ll N,
\end{equation}
already after a very short time \(t\gg |i-i'|/N\). Here we consider the conditional expectation of the
eigenvectors given that the eigenvalues are fixed.
 Since the global equilibrium of the
DBM is the constant function \(f_i=1\), equilibration directly implies \emph{smoothing} or \emph{regularisation} in the 
dependence on the indices \(i\).

On the other hand, by spectral theorem 
\begin{equation}\label{qGq}
  \langle \bq, \Im G(\lambda_i + \ii \eta)\bq\rangle = \frac{1}{N} \sum_{i'=1}^N \frac{\eta}{ (\lambda_i-\lambda_{i'})^2+\eta^2}
   N|\langle \bu_{i'}, \bq\rangle |^2,
\end{equation}
where \(G=G(z)=(W-z)^{-1}\) is the resolvent at a spectral parameter \(z\in \C\setminus\R\).
Using that the eigenvalues \(\lambda_{i'}\) are \emph{rigid}, i.e.\ they are very  close the
corresponding quantiles \(\gamma_{i'}\) of the Wigner semicircle density (see~\eqref{def:Omega} later), 
the \(i'\)-summation in~\eqref{qGq} is  a regularised averaging over indices \(|i'-i|\lesssim N\eta\).
Performing the \(i'\) summation in~\eqref{qGq} by using~\eqref{2ndmoment} we obtain
\[
   \E \big[ N |\langle \bu_i, \bq\rangle |^2 \big| {\bm \lambda}\big]
    \approx \frac{1}{\Im m_{sc}(\gamma_i)}  \E \big[ \langle \bq, \Im G(\gamma_i + \ii \eta)\bq\rangle\big| {\bm \lambda}\big],
\]
for times \(t\gg \eta\) where \(m_{sc}\) is the Stieltjes transform of the Wigner semicircle law.
 Choosing \(\eta\)  slightly above the local eigenvalue spacing, \(\eta=N^{-1+\epsilon}\)  in the bulk of the spectrum, we have
\(\langle \bq, \Im G(\gamma_i + \ii \eta)\bq\rangle\approx \Im m_{sc}(\gamma_i)\) 
not only in expectation but even in high probability by the \emph{isotropic local law} for Wigner matrices~\cite{MR3103909}.
Combining these inputs we obtain \(\E N |\langle \bu_i, \bq\rangle |^2\approx 1\) along the DBM after a short time
 \(t\gg N^{-1+\epsilon}\). A similar argument holds for higher moments.
Finally, the small Gaussian component added by the DBM can be removed by standard perturbation methods, by the
so called  \emph{Green function comparison} theorems.

\subsection{Dyson Brownian motion for general overlaps}
Given the method to handle \(N |\langle \bu_i, \bq\rangle |^2\) described above, 
the Gaussianity of overlaps \(\braket{\bu_i, A \bu_i}\) with a general high rank matrix \(A\) 
can be approached in two natural ways. We now explain  both of them to justify our choice.
 The first approach is to write
 \(A=\sum_{k=1}^N a_k |\bq_k\rangle \langle \bq_k|\) in spectral decomposition
with \(|a_k|\lesssim 1\) and an orthonormal set \(\{ \bq_k\}_{k=1}^N\) to have
\begin{equation}\label{spe}
   \braket{\bu_i, A \bu_i} = \sum_{k=1}^N a_k |\langle \bu_i, \bq_k \rangle|^2.
\end{equation} 
If all overlaps \(|\langle \bu_i, \bq_k \rangle|^2\), \(k=1, 2, \ldots, N\), were independent,  then the central 
limit theorem applied to the summation in~\eqref{spe} would prove the normality 
of  \(\braket{\bu_i, A \bu_i}\). This requires that the number of nonzero summands in~\eqref{spe}, the rank of \(A\), also goes to infinity as \(N\) increases. Hence, via the spectral decomposition of \(A\),
 the Gaussianity
of \(\braket{\bu_i, A \bu_i}\) appears rather an effect of the approximate independence of 
 the overlaps \(|\langle \bu_i, \bq_k  \rangle|^2\) for different \(k\)'s  than their actual limit distribution.
 The analysis of the eigenvector moment flow~\cite{MR3606475,2005.08425}  yields this independence for finitely many \(k\)'s,
 but it is not well suited for tracking  overlaps \(|\langle \bu_i, \bq_k  \rangle|^2\) 
 with a very large number of \(\bq_k\) vectors simultaneously. 
 Hence we discarded this approach.

The second natural approach is to generalise the eigenvector moment flow to moments of \(\braket{\bu_i, A \bu_i}\);
this has been first achieved in~\cite{MR4156609}.
Such  flow naturally  involves off-diagonal overlaps \(\braket{\bu_i, A \bu_j}\) as well.
Therefore, we need to describe conditional 
moments of the form \(\E \big[ \prod_{r=1}^n \braket{\bu_{i_r}, A \bu_{j_r}} \big|{\bm \lambda}\big]\) 
with different collections
of \emph{index pairs} \((i_1, j_1), (i_2, j_2), \ldots, (i_n, j_n)\)
with the constraint that every index appears even number of times.  Thus the relevant moments can naturally be represented
in an \(n\)-dimensional subset \(\Lambda^n\)  of \([N]^{2n}\) (see Section~\ref{sec:equivrep}).
Moreover,  in~\cite[Eq. (2.15)]{MR4156609}  a certain symmetrised linear combination of \(n\)-order moments, 
the \emph{perfect matching observable}  was found that satisfies a closed equation along the Dyson Brownian motion,
see~\eqref{eq:deff} and~\eqref{eq:1dequa}.  Moments of diagonal overlaps \(\braket{\bu_i, A \bu_i}\) can then be recovered 
from the perfect matching observable by setting all indices equal. Off-diagonal overlaps \(\braket{\bu_i, A \bu_j}\)
in general cannot be recovered (except in the \(n=2\) case using an additional  anti-symmetric (``fermionic'') version
of the perfect matching observable~\cite{MR4242625}). 

The main obstacle along this second approach
 is the lack of the analogue of~\eqref{qGq} for general overlaps \(\braket{\bu_i, A \bu_j}\). Consider the \(n=2\) case. A (regularized) local averaging in \emph{one} index yields
\begin{equation}\label{1av}
 \frac{1}{N} \sum_{i'=1}^N \frac{\eta}{ (\lambda_i-\lambda_{i'})^2+\eta^2}
    N|\braket{\bu_{i'}, A \bu_j} |^2 = \braket{\bu_j, A \Im G(\gamma_i + \ii \eta) A \bu_j}
\end{equation}
which still involves an eigenvector \(\bu_j\), hence is not accessible solely by  resolvent methods. 
Note that for \(A=|\bq\rangle  \langle \bq|\) the overlap  \(\braket{\bu_{i'}, A \bu_j}\) factorizes and
the averaging in \(i'\) can be done independently of \(j\).
For  general \(A\) we can handle  an averaging
in \emph{both} indices, i.e.\ we will  use that
\begin{equation}\label{2av}
 \frac{1}{N^2} \sum_{i', j'=1}^N \frac{\eta}{ (\lambda_i-\lambda_{i'})^2+\eta^2}\frac{\eta}{ (\lambda_j-\lambda_{j'})^2+\eta^2}
   N|\braket{\bu_{i'}, A \bu_{j'}} |^2 = \braket{  A \Im G(\gamma_i + \ii \eta) A \Im G(\gamma_j + \ii \eta)}.
\end{equation}
The normalised trace in the right hand side
 is accessible by  resolvent methods using the recent multi-\(G\) local law proven in~\cite[Prop. 3.4]{2012.13215}.
However, the generator of the eigenvector moment flow~\eqref{eq:1dkernel} involves the \emph{sum} of averaging  
operators as in~\eqref{1av} in all
coordinate directions and  not their \emph{product} as needed in~\eqref{2av}.  Higher moments (\(n>2\)) require
averaging in more than two indices
 simultaneously that is not apparently available in the generator.
To remedy this situation, we now review how the equilibration (smoothing) property of
the parabolic equation for the perfect matching observable can be manifested.

\subsection{Local smoothing of the eigenvector moment flow: an overview}
The technically simplest way to exploit the smoothing effect is via the maximum principle introduced in~\cite{MR3606475}. However, this requires that the generator is negative and itself has
the necessary local averaging property to obtain a quantity computable by a local law;
this is the case for eigenfunction overlaps as in~\eqref{qGq} but not for general overlaps  \(\braket{\bu_i, A \bu_j}\) in~\eqref{1av}. We remark that the maximum principle was also used  in~\cite{MR4156609} for more general overlaps, but only for
getting an a priori bound and not for establishing their distribution. For this cruder purpose 
a rougher bound on~\eqref{1av} was sufficient that could be iteratively improved, but always by an \(N^\epsilon\) factor
off  the optimal value. 

A technically much more demanding way to exploit the equilibration of the eigenvector moment flow
would be via homogenisation theory. In random matrix theory
 homogenisation was originally introduced for the Dyson eigenvalue flow in~\cite{MR3541852, MR3914908}
 by noticing that the generator is a discrete approximation of the one dimensional fractional
 Laplacian operator \(|p|=\sqrt{-\Delta}\) with translation invariant kernel \((x-y)^{-2}\) whose heat kernel is explicitly known.
Unfortunately, the eigenvector flow is more complicated and a good approximation with a well-behaving continuous heat
kernel is missing  although homogenisation might also be accessible via a sequence of maximum principles
as in~\cite{1812.10376}.  

Finally, the last and most flexible method for equilibration are the ultracontractivity estimates on the heat kernel that can
be obtained by the standard Nash method from Poincar\'e or Sobolev inequalities for the Dirichlet form
determined by the generator. In random matrix theory, these ideas have been introduced in~\cite{MR3372074} 
for the eigenvalue gap statistics 
and have later been used as a priori bounds for the homogenisation theory.  However, in the bulk regime
they are barely not sufficiently strong to get the necessary precision for individual eigenvalues; they had
to be complemented either by De Giorgi-Nash-Moser  H\"older regularity estimates~\cite{MR3372074} or 
homogenisation~\cite{MR3541852, MR3914908}. 

The  recent work by Marcinek and Yau~\cite{2005.08425} remedies this shortcoming of the ultracontractivity
bound by combining it with an energy method.  The main motivation of~\cite{2005.08425} 
was to consider the joint distribution of the overlaps  \(|\braket{\bu_i, \bq_k}|^2\) for several eigenvectors
and several test vectors simultaneously. The generator of the  resulting \emph{coloured eigenvector moment flow}
lacks the  positivity preserving property rendering the simple  argument via  maximum principle
impossible.  It turns out that 
this lack of positivity is due to a new \emph{exchange term} in the generator that is present only 
because several \(\bq_k\)'s (distinguished by colours) are considered simultaneously. However,  the
generator with the problematic exchange term is still positive in \(L^2\)-sense and its Dirichlet
form satisfies the usual Poincar\'e inequality from which ultracontractivity bounds can still be derived.
The additional smallness now comes from an effective decay of the \(L^2\)-norm of the solution
where local averaging like~\eqref{qGq} can be exploited. 

\subsection{Main ideas  of the current paper} The proof of Gaussianity of~\eqref{eq:flu} consists of three steps.

\begin{enumerate}[label=Step \arabic*.]
\item\label{energy} We use the energy method inspired by~\cite{2005.08425} together with 
the recent two-\(G\) local law from~\cite[Prop. 3.4]{2012.13215}
and more general multi-\(G\) local laws, proven in Section~\ref{sec:llaw}, to exploit 
an effective averaging mechanism  to reduce  the \(L^2\)-norm of the solution. 
In particular, to understand~\eqref{2av} we need a  two-\(G\) local law 
instead of the single-\(G\) isotropic law 
used in~\eqref{qGq}.
\item\label{ultracontractivity} We use an \(L^2\to L^\infty\) ultracontractivity bound of
 the colourblind eigenvector moment flow from~\cite[Proposition 6.29]{2005.08425}.
\item\label{gft step} The  first two steps   prove the Gaussianity of the overlap~\eqref{eq:flu}
for Wigner matrices with a tiny Gaussian component. With a standard Green function comparison
argument combined with the a priori bound~\eqref{eq:eth} proven in~\cite{2012.13215}
   we remove this Gaussian component. 
\end{enumerate}
\ref{ultracontractivity} and~\ref{gft step} are standard adaptations of existing previous results, so we focus only on explaining~\ref{energy}
We use the energy method in a very different way and for a  very different purpose than~\cite{2005.08425}, but 
for the same reason: its  robustness. In the standard energy argument, if \(f_t\) satisfies the 
parabolic evolution equation \(\partial_t f_t = {\mathcal L}_t f_t\) with a (time-dependent)
generator \({\mathcal L}_t\), then 
\[
  \frac{1}{2}\partial_t \| f_t\|_2^2 =\braket{ f_t, {\mathcal L_t}f_t} =: - D_t(f_t)\le 0,
\]
where \(D_t\) is the Dirichlet form (energy) associated to the generator \(\mathcal{L}_t\). The goal is to give a good lower bound 
\begin{equation}\label{Df}
D_t(f)\ge c\|f \|_2^2 - \mbox{error},
\end{equation}
and use a Gronwall argument to conclude an effective \(L^2\)-decay along the
dynamics. However, at this moment,  the Dirichlet form may first be replaced by a smaller one,
\(\widetilde D_t(f)\lesssim D_t(f) \), for which an effective lower bound~\eqref{Df} is easier to obtain. 
In our case, the gain comes from estimating the error term in~\eqref{Df} by exploiting the
local averaging in all directions  as in~\eqref{2av} so that we could use the  multi-\(G\) local law.

How to find \(\widetilde D\)?
Very heuristically, the generator of the eigenvector moment flow is a discrete analogue of \(|p_1|+|p_2|+ \cdots + |p_n|\), i.e.
 the \emph{sum} of \(|p|\)-operators along all the \(n\) coordinate directions in the  \(n\)-dimensional space \(\Lambda^n\).
 However, the necessary averaging in~\eqref{2av}
 is rather the product of these one dimensional operators.
  Normally, sums of first order differential
 operators cannot be compared with their product since they scale differently with the length. But our
 operators have a  short range regularization on the scale \(\eta\), i.e.\ they rather correspond to \(\eta^{-1}[1-e^{-\eta|p|}]\) than just \(|p|\) (see~\cite[Theorem 7.12]{MR1817225}). Therefore, we will prove the discrete analogue of the operator inequality
 \begin{equation}\label{replace}
  \frac{1}{\eta} \prod_{r=1}^n \left(1-e^{-\eta|p_r|}\right)\le C(n)  \sum_{r=1}^n \frac{1}{\eta}\left[1-e^{-\eta|p_r|	}\right]
 \end{equation}
 on \(\R^n\) and their quadratic forms will be the two Dirichlet forms \(\widetilde D\) and \(D\).
 Since the generator of \(\widetilde D\) now averages in all directions, these 
 averages yield traces of products \(\Im G A \Im G A \ldots \Im G A\) for which we have
 a good local law, hence the corresponding error in~\eqref{Df} is smaller than its naive a priori bound
 using only~\eqref{eq:eth}. This crucial gain provides the  additional smallness to 
 overcome the  general fact that  ultracontractivity bounds alone are barely not sufficient
 to gain sufficiently precise information on individual eigenvalues and eigenvectors in the bulk.
 
 The actual proof requires several technical steps such as (i) localising
 the dynamics  by considering a short range approximation and treating the 
 long range part as a perturbation; (ii) finite speed of propagation for the short range dynamics;
 (iii) cutoff the initial data in a smooth way so that cutoff and time evolution almost commute.
Since these steps have appeared in the literature  earlier, we will not reprove them here,
  we just refer to~\cite{2005.08425} where they have been adapted to the eigenvector moment flow.
 We will give full details only for Step 1.    

Parallel with but independently of the current work, Benigni and Lopatto~\cite{2103.12013} have proved the CLT for $\braket{{\bm u}_i,A{\bm u}_i}$ for the observable $A$ projecting onto a deterministic set of orthonormal vectors $A=\sum_{\alpha\in I} |q_\alpha\rangle\langle q_\alpha|$, with $N^\epsilon\le |I|\le N^{1-\epsilon}$, for some small fixed $\epsilon>0$. Their low rank assumption is complementary to our condition $\braket{\mathring{A}^2}\ge c$ for this class of projection operators, moreover their result also covered the edge regime. The low rank assumption allowed them to operate with the eigenvector moment flow from \cite{MR4242625, MR4156609}. However, their control can handle overlaps with at most $N^{1-\epsilon}$ vectors $q_\alpha$ simultaneously. It seems that this approach has a natural limitation preventing it from using it for high rank observables, e.g. for $|I|\sim N$. In contrast, we consider overlaps $\braket{{\bm u}_i,A{\bm u}_j}$ directly without relying on the spectral decomposition of $A$.

\subsection*{Notation and conventions}
We introduce some notations we use throughout the paper. For integers \(k\in\N \) we use the notation \([k]:= \set{1,\ldots, k}\). For positive quantities \(f,g\) we write \(f\lesssim g\) and \(f\sim g\) if \(f \le C g\) or \(c g\le f\le Cg\), respectively, for some constants \(c,C>0\) which depend only on the constants appearing in~\eqref{eq:momentass}. We denote vectors by bold-faced lower case Roman letters \({\bm x}, {\bm y}\in\C ^N\), for some \(N\in\N\). Vector and matrix norms, \(\norm{\vx}\) and \(\norm{A}\), indicate the usual Euclidean norm and the corresponding induced matrix norm. For any \(N\times N\) matrix \(A\) we use the notation \(\braket{ A}:= N^{-1}\Tr  A\) to denote the normalized trace of \(A\). Moreover, for vectors \({\bm x}, {\bm y}\in\C^N\) we define
\[
\braket{ {\bm x},{\bm y}}:= \sum \overline{x}_i y_i.
\]
We will use the concept of ``with very high probability'' meaning that for any fixed \(D>0\) the probability of the \(N\)-dependent event is bigger than \(1-N^{-D}\) if \(N\ge N_0(D)\). Moreover, we use the convention that \(\xi>0\) denotes an arbitrary small 
positive constant which is independent of \(N\).

\subsection*{Acknowledgement} L.E. would like to thank Zhigang Bao for many illuminating discussions in an 
early stage of this research. The authors are also grateful to Paul Bourgade for his comments on the manuscript and the anonymous referee for several useful suggestions.

\section{Main results}

Let \(W\) be an \(N\times N\) real symmetric or complex Hermitian Wigner matrix. We formulate the following assumptions on \(W\).

\begin{assumption}\label{ass:entr}
  We assume that the matrix elements \(w_{ab}\) are independent up to the Hermitian symmetry \(w_{ab}=\overline{w_{ba}}\) and identically distributed in the sense that \(w_{ab}\stackrel{\mathrm{d}}{=} N^{-1/2}\chi_{\mathrm{od}}\), for \(a<b\), \(w_{aa}\stackrel{\mathrm{d}}{=}N^{-1/2} \chi_{\mathrm{d}}\), with \(\chi_{\mathrm{od}}\) being a real or complex random variable and \(\chi_{\mathrm{d}}\) being a real random variable such that \(\E \chi_{\mathrm{od}}=\E \chi_{\mathrm{d}}=0\) and \(\E |\chi_{\mathrm{od}}|^2=1\). In the complex case we also assume that \(\E \chi_{\mathrm{od}}^2=0\). In addition, we assume the existence of the high moments of \(\chi_{\mathrm{od}}\), \(\chi_{\mathrm{d}}\), i.e.\ that there exist constants \(C_p>0\), for any \(p\in\N \), such that
  \begin{equation}
  \label{eq:momentass}
    \E \abs*{\chi_{\mathrm{d}}}^p+\E \abs*{\chi_{\mathrm{od}}}^p\le C_p.
  \end{equation}
\end{assumption}

 Let \(\lambda_1\le \lambda_2\le \ldots \le \lambda_N\) be its eigenvalues in 
increasing order and 
 denote by \({\bm u}_1,\dots, {\bm u}_N\) the corresponding orthonormal eigenvectors.  
 For any \(N\times N\) matrix \(A\) we denote by \(\mathring{A}:= A-\braket{A}\)  the traceless part  of \(A\). We now state our main result.

\begin{theorem}[Central Limit Theorem in the QUE]\label{theo:flucque} Let \(W\)  be a  real symmetric (\(\beta=1\)) or complex Hermitian (\(\beta=2\)) 
Wigner matrix satisfying Assumptions~\eqref{ass:entr}.
Fix small \(\delta,\delta'>0\) and let \(A=A^*\)  be a deterministic \(N\times N\) matrix with \(\norm{A}\lesssim 1\)
and \(\braket{\mathring{A}^2}\ge \delta'\). In the real symmetric case we also assume that \(A\in\mathbf{R}^{N\times N}\)
is real.  Then for any \(i\in [\delta N, (1-\delta) N]\) it holds
\begin{equation}
\sqrt{\frac{\beta N}{2\braket{\mathring{A}^2}}} \big[\braket{{\bm u}_i,A {\bm u_i}}-\braket{A}\big]\Rightarrow \mathcal{N},\qquad \mbox{as 
\;\; \(N\to\infty\)}
\end{equation}
in the sense of moments, with \(\mathcal{N}\) being a standard real Gaussian random variable. The speed of convergence is explicit, see~\eqref{eq:cltgcomp}.
\end{theorem}

\section{Perfect Matching observables}
For definiteness, we present the
 proof for  the real symmetric case, the analysis for the  complex 
Hermitian case it is completely analogous and so omitted. We only mention that the main difference
between the two symmetry classes is that the perfect matching observables \(f_{\bm\lambda,t}\)
in~\eqref{eq:deff} are defined slightly differently (see~\cite[Eq. (A.3)]{MR4156609}) but the current proof can be easily adapted
to this case.

Consider the matrix flow
\begin{equation}\label{eq:matdbm}
\dif W_t=\frac{\dif \widetilde{B}_t}{\sqrt{N}}, \qquad W_0=W,
\end{equation}
with \(\widetilde{B}_t\) being a standard real symmetric Brownian motion (see e.g.~\cite[Definition 2.1]{MR3606475}). We denote the resolvent of \(W_t\) by \(G=G_t(z):=(W_t-z)^{-1}\), for \(z\in\mathbf{C}\setminus\mathbf{R}\). It is well known (see e.g.~\cite{MR2871147,MR3103909,MR3183577}) that as \(N\to \infty\) the resolvent \((W-z)^{-1}\) becomes approximately deterministic; its deterministic approximation is given by the unique solution of the scalar quadratic equation
\begin{equation}
-\frac{1}{m(z)}=z+m(z), \qquad \Im m(z)\Im z>0.
\end{equation}
In particular, \(m(z)=m_{\mathrm{sc}}(z)\), \(m_{\mathrm{sc}}(z)\) being the Stieltjes transform of the semicircular law \(\rho_{\mathrm{sc}}(x):=(2\pi)^{-1}\sqrt{(4-x^2)_+}\). The deterministic approximation of \(G_t(z)\) is given by \(m_t(z)\), with \(m_t\) the solution of 
\begin{equation}
\partial_t m_t(z)=-m_t\partial_z m_t(z), \qquad m_0=m.
\end{equation}
From now on by \(\rho_t=\rho_t(z)\) we denote \(\rho_t(z):=\pi^{-1}\Im m_t (z)\), for any \(t\ge 0\). In fact, starting from  the  standard semicircle \(\rho_0=\rho_{\mathrm{sc}}\), the density \(\rho_t(x+\ii 0)\) is just a rescaling of \(\rho_0\) by a factor \cred{\(1+t\)}.

By~\cite[Definition 2.2]{MR3606475} it follows that the eigenvectors \({\bm u}_1(t),\dots, {\bm u}_N(t)\) of \(W_t\), corresponding to the eigenvalues \(\lambda_1(t)\le \lambda_2(t)\le \dots\le \lambda_N(t)\), are a solution of the following system of SDE (dropping the time dependence):
\begin{align}\label{eq:evaluflow}
\dif \lambda_i&=\frac{\dif B_{ii}}{\sqrt{N}}+\frac{1}{N}\sum_{j\ne i} \frac{1}{\lambda_i-\lambda_j} \dif t \\\label{eq:evectorflow}
\dif {\bm u}_k&=\frac{1}{\sqrt{N}}\sum_{j\ne i} \frac{\dif B_{ij}}{\lambda_i-\lambda_j}{\bm u}_j-\frac{1}{2N}\sum_{j\ne i} \frac{{\bm u}_i}{(\lambda_i-\lambda_j)^2}\dif t,
\end{align}
with \(\{B_{ij}\}_{i,j\in [N]}\) being a standard real symmetric Brownian motions. See~\cite[Theorem 2.3]{MR3606475} for the existence and uniqueness of the strong solution of~\eqref{eq:evaluflow}--\eqref{eq:evectorflow}.

By~\eqref{eq:evectorflow} it follows that the flow for the diagonal overlaps \(\braket{{\bm u}_i, A {\bm u}_i}\) naturally depends also on the off-diagonal overlap \(\braket{{\bm u}_i, A {\bm u}_j}\), hence our analysis will concern not only diagonal overlaps, but also off-diagonal ones. 
Since \(\braket{{\bm u}_i, A {\bm u}_i}-\braket{A}=\braket{{\bm u}_i, \mathring{A} {\bm u}_i}\) and
\(\braket{{\bm u}_i, A {\bm u}_j}=  \braket{{\bm u}_i, \mathring{A} {\bm u}_j}\) for \(i\ne j\), without loss of generality we
may assume for the rest of the paper, that \(A\) is traceless, \(\braket{A}=0\), i.e.\ \(A=\mathring{A}\). For traceless \(A\)
we introduce the short-hand notation
\begin{equation}
 p_{ij}=p_{ij}(t):=\braket{{\bm u}_i(t), A {\bm u}_j(t)},
\quad i, j\in [N]. 
\end{equation}

We are now ready to write the flow for monomials of  \(p_{ii}\), \(p_{ij}\) (see~\cite[Theorem 2.6]{MR4156609} for the derivation of the flow). For any fixed \(n\), we 
will only need to consider monomials  of the form \(\prod_{k=1}^n p_{i_k j_k}\) where each index appears even number of times;
it turns out that the linear combinations of such monomials with a fixed degree \(n\) are invariant under the flow.

To encode general monomials, we use a particle picture (introduced in~\cite{MR3606475} and developed in~\cite{MR4156609, 2005.08425}) where each particle on 
the set of integers \([N]\) corresponds to two occurrences of an index \(i\) in the monomial product.
We use the same notation as in~\cite{MR4156609} and  we define \({\bm \eta}:[N] \to \mathbf{N}\), where \(\eta_j:={\bm \eta}(j)\) is interpreted as the number of particles at the site \(j\), and \(n({\bm \eta}):=\sum_j \eta_j= n\) denotes the total number of particles that is conserved under the flow. The space of \(n\)-particle configurations is denoted by \(\Omega^n\).
Moreover, for any index pair \(i\ne j\in[N]\), we define \({\bm \eta}^{ij}\) to be the configuration obtained moving a particle from the site \(i\) to the site \(j\), if there is no particle in \(i\) then we define \({\bm \eta}^{ij}={\bm \eta}\). For any configuration \({\bm \eta}\) consider the set of vertices
\begin{equation}
\mathcal{V}_{{\bm \eta}}:=\{(i,a): 1\le i \le n, 1\le a\le 2\eta_i\},
\end{equation}
and let \(\mathcal{G}_{\bm \eta}\) be the set of perfect matchings on \(\mathcal{V}_{{\bm \eta}}\). 
Note that every particle configuration \(\eta\)  gives rise to two vertices in \(\mathcal{V}_{{\bm \eta}}\), thus the elements of 
\(\mathcal{V}_{{\bm \eta}}\) represent the indices in the product \(\prod_{k=1}^n p_{i_k j_k}\).

There is no closed equation for individual products \(\prod_{k=1}^n p_{i_k j_k}\), but  there is one for a certain symmetrized
linear combination, see~\cite[Eq. (2.15)]{MR4156609}. Therefore, for any perfect matching \(G\in \mathcal{G}_{\bm \eta}\) we
 define
\begin{equation}\label{eq:defpg}
P(G):=\prod_{e\in \mathcal{E}(G)}p(e), \qquad p(e):=p_{i_1i_2},
\end{equation}
where \(e=\{(i_1,a_1),(i_2,a_2)\}\in \mathcal{V}_{\bm \eta}\), and \(\mathcal{E}(G)\) denotes the edges of \(G\). 
For example, for \(n=2\) and  for the configuration 
\({\bm \eta}\) defined by \({\bm\eta}(i)=   {\bm\eta}(j)= 1\) with some \(i\ne j\) and zero otherwise, we have three perfect matchings
corresponding to \(p_{ii} p_{jj}\) and twice \(p_{ij}^2\).
For \(n=3\) and \({\bm \eta}\) defined by \({\bm\eta}(i)=   {\bm\eta}(j) = {\bm\eta}(k) =1\), we have 15 perfect matchings;
\(p_{ii} p_{jj} p_{kk}\), two copies \(p_{ij}^2p_{kk}\), \(p_{ik}^2 p_{jj}\), \(p_{jk}^2 p_{ii}\) each and 8 copies of \(p_{ij}p_{jk}p_{ki}\).

We are now  ready to define the \emph{perfect matching observable} for any given configuration \({\bm \eta}\), 
\begin{equation}
  \label{eq:deff}
  f_{{\bm \lambda},t}({\bm \eta}):= \frac{N^{n/2}}{ [2\braket{{A}^2}]^{n/2}} \frac{1}{(n-1)!!}\frac{1}{\mathcal{M}({\bm \eta}) }\E\left[\sum_{G\in\mathcal{G}_{\bm \eta}} P(G)\Bigg| 
  {\bm \lambda}\right], \quad \mathcal{M}({\bm \eta}):=%
  \prod_{i=1}^N (2\eta_i-1)!!,
  \end{equation}
with \(n\) being the number of particles in the configuration \({\bm \eta}\). Here we took the conditioning on the entire flow of eigenvalues, \({\bm \lambda} =\{\bm \lambda(t)\}_{t\in [0,T]}\) for some  fixed  \(T>0\). From now on we will always assume that $T\ll 1$ (even if not stated explicitly). The observable \(f_{\bm\lambda,t}\) satisfies a parabolic partial differential equation, see~\eqref{eq:1dequa} below.
\begin{remark}
For any \(k\in \mathbf{N}\) the double factorial \(k!!\) is defined by \(k!!=k(k-2)!!\), \(1!!=0!!=(-1)!!=1\). We remark that in~\cite{MR4156609, 2005.08425} the authors use a different convention for the double factorial, i.e.\ in these papers \(k!!=(k-1)(k-2)!!\).
\end{remark}

Note that \(f\) in~\eqref{eq:deff} is  defined slightly differently
 compared to the definition in~\cite[Eq. (2.15)]{MR4156609}, where the authors do not have the additional 
 \((N/(2\braket{{A}^2})^{n/2}[(n-1)!!]^{-1}\) factor. Our normalisation factor is dictated by
 the principle that for traceless \(A\) we expect \(\sqrt{N}p_{ii}=\sqrt{N}[\langle {\bm u}_i, A{\bm u}_i\rangle ]\) to be approximately
 a centred normal random variable with variance \(2\braket{{A}^2}\). In particular the \(n\)-th moment of
 \((N/ 2\braket{{A}^2})^{1/2}p_{ii}\) for even \(n\) is close to \((n-1)!!\). 
  Therefore if \({\bm\eta}\) is a configuration with \(n\)
 particles all sitting at the same site \(i\),  i.e.\ \({\bm \eta}(i)=n\) and zero otherwise, then 
 \(\mathcal{M}({\bm \eta})= (2n-1)!!\) is the number of perfect matchings and therefore
 we expect \(f_{{\bm \lambda},t}({\bm \eta})  \approx 1\).
 Note that using the a priori bound \(|p_{ij}|\le N^{-1/2+\xi}\), for any \(\xi>0\), proven in~\cite[Theorem 2.2]{2012.13215}
    we have \(|f_{{\bm \lambda},t}|\lesssim N^\xi\) with very high  probability,  while the analogous 
    quantity \(f_{{\bm \lambda},t}\) defined  in~\cite[Eq. (2.15)]{MR4156609} has  an a priori bound  of order \(N^{-n/2+\xi}\).

We always assume that the entire eigenvalue trajectory 
\(\{\bm \lambda(t)\}_{t\in [0,T]}\) satisfies the usual rigidity estimate (see e.g.~\cite[Theorem 7.6]{MR3068390} or~\cite{MR2871147}). More precisely,  for any fixed \(\xi>0\) we define
\begin{equation}\label{def:Omega}
  \Omega=\Omega_\xi:= \Big\{ \sup_{0\le t \le T} \max_{i\in[N]} N^{2/3}\widehat{i}^{1/3} | \lambda_i(t)-\gamma_i(t)| \le N^\xi\Big\}
\end{equation}
where \(\widehat{i}:=i\wedge (N+1-i)\),
then we have
\[
  \mathbf{P} (\Omega_\xi)\ge 1-  C(\xi, D) N^{-D}
\]
for any (small) \(\xi>0\) and (large) \(D>0\). 
 Here \(\gamma_i(t)\) are the classical eigenvalue locations (\emph{quantiles}) defined by
\begin{equation}\label{eq:quantin}
  \int_{-\infty}^{\gamma_i(t)} \rho_t(x)\, \dif x=\frac{i}{N}, \qquad i\in [N],
\end{equation}
where \(\rho_t(x)= \frac{1}{2(1+t)\pi}\sqrt{(4(1+t)^2-x^2)_+}\) is the semicircle law corresponding to \(W_t\).
 Note that \(|\gamma_i(t)-\gamma_i(s)|\lesssim |t-s|\) in the bulk, for any \(t,s\ge 0\), as a consequence of the smoothness of \(t\to\rho_t\) in the bulk.

By~\cite[Theorem 2.6]{MR4156609} we have that
\begin{align}\label{eq:1dequa}
\partial_t f_{{\bm \lambda},t}&=\mathcal{B}(t)f_{{\bm \lambda},t}, \\\label{eq:1dkernel}
\mathcal{B}(t)f_{{\bm \lambda},t}&=\sum_{i\ne j} c_{ij}(t) 2\eta_i(1+2\eta_j)\big(f_{{\bm \lambda},t}({\bm \eta}^{kl})-f_{{\bm \lambda},t}({\bm \eta})\big).
\end{align}
where 
\begin{equation}\label{eq:defc}
c_{ij}(t):= \frac{1}{N(\lambda_i(t) -  \lambda_j(t))^2}.
\end{equation}
Note that \(c_{ij}\) depends on \(\{{\bm \lambda}(t)\}_{t\in [0, T]}\), for some \(T>0\), but we omit this fact from the notation.
We note that this flow was originally derived  for special observables given in~\cite[Eq. (2.6)]{MR4156609},  but the  same derivation immediately holds 
for arbitrary \(A\) (see~\cite[Remark 2.8]{MR4156609}).

The main technical ingredient that will be used in the proof of Theorem~\ref{theo:flucque} 
 is the following proposition, whose proof is postponed to Section~\ref{sec:DBM}.

\begin{proposition}\label{pro:flucque}
For any \(n\in\mathbf{N}\) there exists \(c(n)>0\) such that for any \(\epsilon>0\), and for any \(T\ge N^{-1+\epsilon}\) it holds
\begin{equation}\label{eq:mainbthissec}
\sup_{{\bm \eta}}\big|f_T({\bm\eta})-\bm1(n\,\, \mathrm{even})\big|\lesssim N^{-c(n)},
\end{equation}
with very high probability, where the supremum is taken over configurations \({\bm \eta}\) such that \(\sum_i \eta_i=n\) and \(\eta_i=0\) for \(i\notin [\delta N, (1-\delta) N]\), with \(\delta>0\) from Theorem~\ref{theo:flucque}. The implicit constant in~\eqref{eq:mainbthissec} depends on \(n\), \(\epsilon\), \(\delta\).
\end{proposition}

\begin{proof}[Proof of Theorem~\ref{theo:flucque}]
We fix \(i\in [\delta N,(1-\delta) N]\) and we choose \({\bm \eta}\) to be the configuration \({\bm \eta}\) with \(\eta_i=n\) and all other \(\eta_j=0\). Then all the terms \(P(G)\) are equal to \(p_{ii}^n\) in the definition of \(f\), see~\eqref{eq:deff}. Then, using~\eqref{eq:mainbthissec} for this particular \({\bm \eta}\), we conclude that 
\begin{equation}\label{eq:cltgcomp}
\E\left[\sqrt{\frac{N}{2\braket{A^2}}}\braket{{\bm u}_i(T),A{\bm u}_i(T)}\right]^n=\bm1(n\,\, \mathrm{even})(n-1)!!+\mathcal{O}\left(N^{-c(n)}\right),
\end{equation}
for any \(i\in [\delta N,(1-\delta) N]\) and \(T\gg N^{-1}\), where we used that \(\norm{f_{T}}_\infty\le N^{n/2}\) deterministically on the complement of the high probability set on which~\eqref{eq:mainbthissec} holds. With~\eqref{eq:cltgcomp} we have proved that Theorem~\ref{theo:flucque} holds for Wigner matrices with a small Gaussian component. For the general case, Theorem~\ref{theo:flucque} follows from~\eqref{eq:cltgcomp} and a standard application of the Green function comparison theorem (GFT), relating the eigenvectors/eigenvalues of \(W_T\) to those of \(W\); see Appendix~\ref{app:GFT} where we recall the argument for completeness.
\end{proof}

\section{DBM analysis}\label{sec:DBM}
In this section we focus on the analysis of the eigenvector moment flow~\eqref{eq:1dequa}--\eqref{eq:1dkernel}. Since in our proof we use some results proven in~\cite{2005.08425}, we start giving an equivalent representation of~\eqref{eq:deff} which is the same used in~\cite{2005.08425} without distinguishing the several colours.

\subsection{Equivalent representation of the flow}\label{sec:equivrep}
Fix \(n\in\mathbf{N}\), then in the remainder of this section we will consider configurations \({\bm \eta}\in\Omega^n\), i.e.\ such that \(\sum_j \eta_j=n\). Following~\cite{2005.08425} (but without the extra complication involving  colours), 
we now give an equivalent representation of the flow~\eqref{eq:1dequa}--\eqref{eq:1dkernel}
which will be defined on the \(2n\)-dimensional lattice \([N]^{2n}\) instead of configurations 
of \(n\) particles. Let \({\bm x}\in [N]^{2n}\) and define
\begin{equation}
n_i({\bm x}):=|\{a\in [2n]:x_a=i\}|,
\end{equation}
for all \(i\in \mathbf{N}\). We define the configuration space
\[ 
  \Lambda^n:= \{ {\bm x}\in [N]^{2n} \, : \,\mbox{\(n_i({\bm x})\) is even for every \(i\in [N]\)} \big\}.%
\]
Note that \(\Lambda^n\) is an \(n\)-dimensional subset of the \(2n\) dimensional
lattice \([N]^{2n}\) in the sense that \(\Lambda^n\) is a finite union of \(n\)-dimensional sublattices of \([N]^{2n}\).
From now on we will only consider configurations \({\bm x}\in\Lambda^n\). In particular, in this representation to each particle is associated a label $a\in [2n]$, i.e.\ there is a particle at a site $i\in [N]$ iff there exists $a\in [2n]$ such that $x_a=i$. Additionally, by the definition of $\Lambda^n$ it follows that the number of particles at a site $i\in [N]$ is always even.

\begin{remark}
Note that in~\cite{2005.08425} the authors consider \({\bm x}\) to be an \(n\)-dimensional vector that lives in the \(n/2\)-dimensional subset \(\Lambda^n\). For notational simplicity, in the current paper we assume that \({\bm x}\) is a \(2n\)-dimensional vector and that \(\Lambda^n\) is \(n\)-dimensional.
\end{remark}

The natural correspondence between the two representations is given by 
 \begin{equation}\label{xeta}
   {\bm \eta} \leftrightarrow {\bm x}\qquad \eta_i=\frac{n_i( {\bm x})}{2}.
\end{equation}
Note that \({\bm x}\) uniquely determines \({\bm \eta}\), but \({\bm \eta}\) determines only the coordinates
of \({\bm x}\) as a multi-set and not its ordering. As an example, the configuration with single (or doubled) particles in \(i_1\ne i_2\) corresponds to six \(\bm x\in\Lambda^2\) as in
\begin{align*}
  \underbrace{\begin{tikzpicture}
    [every node/.style={fill, circle, inner sep = 1.5pt},baseline=0pt]
  \draw (-.5,0) -- (1.5,0);
  \node[label=below:$i_1$] (i1) at (0,0) {};
  \node[label=below:$i_2$] (i2) at (1,0) {};
  \end{tikzpicture}}_{\bm\eta\text{-repr.}}\quad \Leftrightarrow \quad 
  \underbrace{\begin{tikzpicture}
    [every node/.style={fill, circle, inner sep = 1.5pt},baseline=0pt]
  \draw (-.5,0) -- (1.5,0);
  \node[label=below:$i_1$] (i1) at (0,0) {};
  \node[label=below:$i_2$] (i2) at (1,0) {};
  \node (i11) at ($(i1)+(0,.33)$) {};
  \node (i21) at ($(i2)+(0,.33)$) {};
  \end{tikzpicture}}_{\text{doubled }\bm\eta\text{-repr.}}\quad \Leftrightarrow \quad 
  \underbrace{\begin{pmatrix}i_1\\ i_1\\ i_2\\ i_2\end{pmatrix}\equiv
  \begin{pmatrix}i_1\\ i_2\\ i_1\\ i_2\end{pmatrix}\equiv
  \begin{pmatrix}i_1\\ i_2\\ i_2\\ i_1\end{pmatrix}\equiv\cdots}_{\bm x\text{-repr.}}.
\end{align*}
Let \(\phi\colon\Lambda^n\to \Omega^n\), \(\phi({\bm x})={\bm \eta}\) 
denote the map that projects the \({\bm x}\)-configuration space to the \({\bm \eta}\)-configuration space using~\eqref{xeta}.
This map naturally pulls back functions \(f\) of \({\bm \eta}\)  to functions of \({\bm x}\)
\[
     (\phi^* f)({\bm x}) = f(\phi({\bm x})).
\]
 We will always consider
functions \(g\) on \([N]^{2n}\) that are push-forwards of some function \(f\) on \(\Omega^n\), 
\(g= f\circ \phi\), i.e.
they correspond to functions on the configurations
\[
   f({\bm \eta})= f(\phi({\bm x}))= g({\bm x}).
\]
In particular \(g\) is supported on \(\Lambda^n\) and it is equivariant under permutation of the 
arguments, i.e.\ it
 depends on \({\bm x}\) only as a multiset. %
We  therefore consider the observable
\begin{equation}\label{eq:defg}
g_{{\bm \lambda},t}({\bm x}):= f_{{\bm \lambda},t}( \phi({\bm x}))
\end{equation}
where \( f_{{\bm \lambda},t}\) was defined in~\eqref{eq:deff}.
 In the following we will often use the notation \(g_t({\bm x})=g_{{\bm \lambda},t}({\bm x})\), dropping the dependence of \(g_t({\bm x})\) on the eigenvalues.

The flow~\eqref{eq:1dequa}--\eqref{eq:1dkernel} can be written in the \({\bm x}\)-representation as follows:
\begin{align}\label{eq:g1deq}
\partial_t g_t({\bm x})&=\mathcal{L}(t)g_t({\bm x}) \\\label{eq:g1dker}
\mathcal{L}(t):=\sum_{j\ne i}\mathcal{L}_{ij}(t), \quad \mathcal{L}_{ij}(t)g({\bm x}):&= c_{ij}(t) \frac{n_j({\bm x})+1}{n_i({\bm x})-1}\sum_{a\ne b\in[2 n]}\big(g({\bm x}_{ab}^{ij})-g({\bm x})\big),
\end{align}
where
\begin{equation}
\label{eq:jumpop}
{\bm x}_{ab}^{ij}:={\bm x}+\delta_{x_a i}\delta_{x_b i} (j-i) ({\bm e}_a+{\bm e}_b),
\end{equation} 
with \({\bm e}_a(c)=\delta_{ac}\), \(a,c\in [2n]\). Clearly this flow preserves the equivariance of \(g\),
i.e.\ it is a map on functions defined on \(\Lambda^n\). The jump operator ${\bm x}_{ab}^{ij}$ defined in~\eqref{eq:jumpop} changes \(x_a,x_b\) from \(i\) to \(j\) if \(x_a=x_b=i\) and otherwise leaves \(\bm x\) unchanged. In the particle picture ${\bm \eta}$ this corresponds in moving one particle from the site $i$ (if there is any) to the site $j$, see the following example for \(n=2\) (with \(i=i_1,j=j_1\) and \(a=1,b=2\)): 
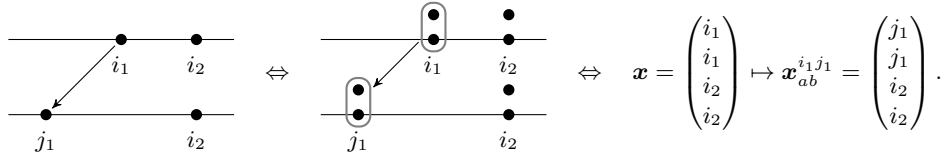
\[ \begin{tikzpicture}
  [every node/.style={fill, circle, inner sep = 1.5pt},baseline=-15pt,>=stealth']
\draw (-1.5,0) -- (1.5,0);
\draw (-1.5,-1) -- (1.5,-1);
\node[label=below:$i_1$] (i1) at (0,0) {};
\node[label=below:$i_2$] (i2) at (1,0) {};
\node[label=below:$j_1$] (j1) at (-1,-1) {};
\node[label=below:$i_2$] (j2) at (1,-1) {};
\draw [shorten >=1pt,shorten <=1pt,->] (i1) -- (j1);
\end{tikzpicture}\quad \Leftrightarrow \quad 
\begin{tikzpicture}
  [every node/.style={fill, circle, inner sep = 1.5pt},baseline=-15pt,>=stealth']
\draw (-1.5,0) -- (1.5,0);
\draw (-1.5,-1) -- (1.5,-1);
\node[label=below:$i_1$] (i1) at (0,0) {};
\node[label=below:$i_2$] (i2) at (1,0) {};
\node[label=below:$j_1$] (j1) at (-1,-1) {};
\node[label=below:$i_2$] (j2) at (1,-1) {};
\node (i1p) at ($(i1)+(0,.33)$) {};
\node (i2p) at ($(i2)+(0,.33)$) {};
\node (j1p) at ($(j1)+(0,.33)$) {};
\node (j2p) at ($(j2)+(0,.33)$) {};
\draw[thick,gray,rounded corners]  ($(i1p.north west)+(-0.1,0.1)$) rectangle ($(i1.south east)+(0.1,-0.1)$);
\draw[thick,gray,rounded corners]  ($(j1p.north west)+(-0.1,0.1)$) rectangle ($(j1.south east)+(0.1,-0.1)$);
\draw [shorten >=8pt,shorten <=8pt,->] ($(i1)!0.5!(i1p)$) -- ($(j1)!0.5!(j1p)$);
\end{tikzpicture}\quad \Leftrightarrow\quad
\bm x=\begin{pmatrix}i_1\\ i_1\\ i_2\\ i_2\end{pmatrix}\mapsto\bm x_{ab}^{i_1j_1}=\begin{pmatrix}j_1\\ j_1\\ i_2\\ i_2\end{pmatrix}.
\]

Define the measure
\begin{equation}\label{eq:revmeasure}
\pi({\bm x}):=\prod_{i=1}^N ((n_i({\bm x})-1)!!)^2
\end{equation}
on \(\Lambda^n\) and the corresponding \(L^2(\Lambda^n)=L^2(\Lambda^n,\pi)\) space equipped with the scalar product
\begin{equation}\label{eq:scalpro}
\braket{f, g}_{\Lambda^n}=\braket{f, g}_{\Lambda^n, \pi}:=\sum_{{\bm x}\in \Lambda^n}\pi({\bm x})\bar f({\bm x})g({\bm x}).
\end{equation}
We will often drop the dependence on the measure \(\pi\) in the scalar product. We also define the following norm on \(L^p(\Lambda^n)\):
\begin{equation}
\norm{f}_p:=\left(\sum_{{\bm x}\in \Lambda^n}\pi({\bm x})|f({\bm x})|^p\right)^{1/p}.
\end{equation}
The  measure \(\pi({\bm x})\) clearly satisfies
\begin{equation}\label{eq:boundsrevmeasure}
1\le \pi({\bm x}) \le (2n-1)!!,
\end{equation}
uniformly in \({\bm x}\in\Lambda^n\). 
A direct calculation in~\cite[Appendix A.2]{2005.08425}
shows that the operator \(\mathcal{L}=\mathcal{L}(t)\)  is symmetric with respect to the measure \(\pi\)
and it is a negative operator on the space \(L^2(\Lambda^n)\) with Dirichlet form
\[
   D(g)=\braket{g, (-\mathcal{L}) g}_{\Lambda^n} = \frac{1}{2}  \sum_{{\bm x}\in \Lambda^n}\pi({\bm x})
   \sum_{i\ne j} c_{ij}(t) \frac{n_j({\bm x})+1}{n_i({\bm x})-1}
   \sum_{a\ne b\in[2 n]}\big|g({\bm x}_{ab}^{ij})-g({\bm x})\big|^2.
\]
We will often omit the time dependence of the generator \(\mathcal{L}(t)\). We denote by \(\mathcal{U}(s,t)\) 
 the semigroup associated to \(\mathcal{L}\) from~\eqref{eq:g1dker}, i.e.\ for any \(0\le s\le t\) it holds
\[
\partial_t\mathcal{U}(s,t)=\mathcal{L}(t)\mathcal{U}(s,t), \quad \mathcal{U}(s,s)=I.
\]

\subsection{Short-range approximation}

Before proceeding we introduce a localised version of~\eqref{eq:g1deq}--\eqref{eq:g1dker}. 
Choose 
an (\(N\)-dependent) parameter \(1\ll K\le \sqrt{N}\) 
 and define the \emph{averaging operator} as a simple multiplication operator by a ``smooth'' cut-off function:
\begin{equation}
\Av(K,{\bm y})h({\bm x}):=\Av({\bm x};K,{\bm y})h({\bm x}), \qquad \Av({\bm x}; K, {\bm y}):=\frac{1}{K}\sum_{j=K}^{2K-1} \bm1(\norm{{\bm x}-{\bm y}}_1<j),
\end{equation}
with \(\norm{{\bm x}-{\bm y}}_1:=\sum_{a=1}^{2n} |x_a-y_a|\). 
While it was denoted and called  averaging operator in~\cite{MR4156609, 2005.08425}, it is rather a \emph{localization}, i.e.\ a multiplication
by a ``smooth'' cutoff function \({\bm x}\to  \Av({\bm x}; K, {\bm y})\) which 
 is centered at \({\bm y}\) and has a soft  range  of size \(K\). The parameters \(K, {\bm y}\) are considered fixed
 and often omitted from the notation. In particular, throughout the paper we will assume that \({\bm y}\) is supported in the bulk, i.e. we will always assume that ${\bm y}\in \mathcal{J}$ (see the definition of $\mathcal{J}$ in \eqref{eq:defintJ} below).

Now we define a short range version of the dynamics~\eqref{eq:g1deq}. Fix an  integer \(\ell\) with \(1\ll\ell\ll K\)
and define the short range coefficients
\begin{equation}\label{eq:ccutoff}
c_{ij}^{\mathcal{S}}(t):=\begin{cases}
c_{ij}(t) &\mathrm{if}\,\, i,j\in \mathcal{J} \,\, \mathrm{and}\,\, |i-j|\le \ell \\
0 & \mathrm{otherwise},
\end{cases}
\end{equation}
where \(c_{ij}(t)\) is defined in~\eqref{eq:defc}. Here
\begin{equation}\label{eq:defintJ}
\mathcal{J}=\mathcal{J}_\delta:=\{ i\in [N]:\, \gamma_i(0)\in \mathcal{I}_\delta\}, \qquad \mathcal{I}_\delta:=(-2+\delta,2-\delta)
\end{equation}
with \(\delta>0\) from Theorem~\ref{theo:flucque}, so that \(\mathcal{I}_\delta\) lies entirely in the bulk spectrum.

We define \(h_t({\bm x})\) as the time evolution  of a localized initial data \(g_0\) 
by the short range dynamics: %
\begin{equation}\label{g-1}
\begin{split}
h_0({\bm x};\ell, K,{\bm y})=h_0({\bm x};K,{\bm y}):&=\Av({\bm x}; K,{\bm y})(g_0({\bm x})-\bm1(n \,\, \mathrm{even})), \\\partial_t h_t({\bm x}; \ell, K,{\bm y})&=\mathcal{S}(t) h_t({\bm x}; \ell, K,{\bm y}),
\end{split}
\end{equation}
where
\begin{equation}\label{g-2}
\mathcal{S}(t):=\sum_{j\ne i}\mathcal{S}_{ij}(t), \quad \mathcal{S}_{ij}(t)h({\bm x}):=c_{ij}^{\mathcal{S}}(t)\frac{n_j({\bm x})+1}{n_i({\bm x})-1}\sum_{a\ne b\in [2n]}\big(h({\bm x}_{ab}^{ij})-h({\bm x})\big).
\end{equation}
Here we used the notation \(h({\bm x})=h({\bm x}; \ell, K,{\bm y})\) to indicate all relevant parameters:
  \(\ell\) indicates
the  short range  of the dynamics, \({\bm y}\) is the centre 
and \(K\) is the range of the cut-off  in the initial condition, and we always choose \(\ell \ll K\).
In~\eqref{g-1} we already subtracted \(\bm1(n \,\, \mathrm{even})\) since in our application the initial condition \(g_0({\bm x})\)
after some local averaging will be close to \(\bm1(n \,\, \mathrm{even})\), hence, after longer time we expect that \(h_t\) tends to zero since
the dynamics has a smoothing effect and it is an \(L^1\) contraction.

\subsection{\texorpdfstring{\(L^2\)}{L2}-bound}\label{sec:l2}
Define the distance on \(\Lambda^n\) as
\begin{equation}
d({\bm x}, {\bm y}):=\sup_{a\in [2n]}|\mathcal{J}\cap \big[\min(x_a,y_a), \max(x_a,y_a)\big)|,
\end{equation}
with \(\mathcal{J}\) defined in~\eqref{eq:defintJ}. Note that \(d\) is not a metric since it is degenerate, but it still symmetric and satisfies the triangle inequality~\cite[Eq.~(5.6)]{2005.08425}. The key ingredient to prove the \(L^2\)-bound in~\eqref{eq:l2b} below is to show that the short range dynamics~\eqref{g-1}--\eqref{g-2} is close to the original dynamics~\eqref{eq:g1deq}--\eqref{eq:g1dker}. This will be achieved using the following finite speed of propagation estimate, proven in~\cite[Theorem 2.1, Lemma 2.4]{MR3690289},~\cite[Proposition 5.2]{2005.08425} (see also~\cite[Eq. (3.15)]{MR4156609}), for \(\mathcal{U}_{\mathcal{S}}(s,t)=\mathcal{U}_{\mathcal{S}}(s,t;\ell)\), which is the transition semigroup associated to the short range generator \(\mathcal{S}(t)\). For any \({\bm x}\in\Lambda^n\) define
the ``delta-function'' on \(\Lambda^n\) as 
\[
\delta_{{\bm x}}({\bm u}):=\begin{cases}
\pi({\bm x})^{-1} &\mathrm{if} \,\, {\bm u}={\bm x} \\
0 &\mathrm{otherwise},
\end{cases}
\]
and denote the matrix entries of \(\mathcal{U}_{\mathcal{S}}(s,t)\) by \(\mathcal{U}_{\mathcal{S}}(s,t)_{{\bm x}{\bm y}}:=\braket{\delta_{{\bm x}}, \mathcal{U}_{\mathcal{S}}(s,t) \delta_{{\bm y}}}\).

\begin{proposition}\label{pro:finitespeed}
Fix any small \(\epsilon>0\), and \(\ell\ge N^\epsilon\). Then for any \({\bm x}, {\bm y}\in \Lambda^n\) with \(d({\bm x}, {\bm y})>N^\epsilon\ell\) it holds
\begin{equation}
\sup_{0\le s_1\le s_2\le s_1+\ell N^{-1}} |\mathcal{U}_{\mathcal{S}}(s_1,s_2;\ell)_{{\bm x}{\bm y}}|\le e^{-N^\epsilon/2},
\end{equation}
on the very high probability event \(\Omega\).
\end{proposition}
This finite speed of propagation together with the fact that the initial condition  \(h_0\) is localized in 
a \(K\)-neighbourhood of a fixed center \({\bm y}\) implies that \(h_t\) is supported in a \(K+N^\epsilon\ell \le 2K\) 
neighbourhood of \({\bm y}\) up an exponentially small tail part.

Using Proposition~\ref{pro:finitespeed}, by~\cite[Corollary 5.3]{2005.08425}, we immediately conclude the following lemma.

\begin{lemma}\label{lem:exchav}
For any times \(s_1,s_2\) such that \(0\le s_1\le s_2\le s_1+\ell N^{-1}\), and for any \({\bm y}\in \Lambda^n\) supported on \(\mathcal{J}\) (i.e.\ \({\bm y}_a\in \mathcal{J}\) for any \(a\in [2n]\)) 
for the commutator of the evolution \(\mathcal{U}_\mathcal{S}\)   and the averaging operator we have
\begin{equation}\label{eq:comm}
\norm{[\mathcal{U}_\mathcal{S}(s_1,s_2;\ell), \Av({\bm y},K)]}_{\infty,\infty}\le C(n)\frac{N^\epsilon\ell}{K},
\end{equation}
for some constant \(C(n)>0\) and for any small \(\epsilon>0\), on the very high probability event \(\Omega\).
\end{lemma}

Another straightforward application of the finite speed of propagation estimate in Proposition~\ref{pro:finitespeed} is the following bound \(\mathcal{U}(s_1,s_2)-\mathcal{U}_{\mathcal{S}}(s_1,s_2;\ell)\). This result was proven in~\cite[Proposition 5.7]{2005.08425} for a specific \(f\) but the same proof applies  for a general function \(f\).

\begin{lemma}\label{lem:shortlongapprox}
Let \(0\le s_1\le s_2\le s_1+\ell N^{-1}\), and \(f\) is a function on \(\Lambda^n\), 
then for any \({\bm x}\in \Lambda^n\) supported on \(\mathcal{J}\) it holds
\begin{equation}\label{eq:shortlong}
\Big| (\mathcal{U}(s_1,s_2)-\mathcal{U}_{\mathcal{S}}(s_1,s_2;\ell) ) f({\bm x}) \Big|\lesssim N^{1+n\xi}\frac{s_2-s_1}{\ell} \| f\|_\infty,
\end{equation}
for any small \(\xi>0\).
\end{lemma}
\begin{proof}
Using Proposition~\ref{pro:finitespeed}, the proof of~\eqref{eq:shortlong} is completely analogous to the proof of~\cite[Proposition 5.7]{2005.08425}, since the only input used in~\cite[Proposition 5.7]{2005.08425} is that
\[
\sum_{j: |j-i|>\ell}\frac{1}{N(\lambda_i-\lambda_j)^2}\le \frac{N^{1+\xi}}{\ell}
\]
on \(\Omega\), which follows by rigidity.
\end{proof}

Before stating the main result of this section we define the set \(\widehat{\Omega}\) on which the local laws for certain products of resolvents and traceless matrices \(A\) hold, i.e.\ for a small \(\omega>2\xi>0\) we define 
 \begin{equation}  
 \label{eq:hatomega}
 \begin{split}
\widehat{\Omega}=\widehat{\Omega}_{\omega, \xi}:&=\bigcap_{\substack{z_i: \Re z_i\in [-3,3], \atop |\Im z_i|\in [N^{-1+\omega},10]}}\Bigg[\bigcap_{k=3}^n \left\{\sup_{0\le t \le T}(\rho_t^*)^{-1/2}\big|\braket{G_t(z_1)A\dots G_t(z_k)A}\big|\le \frac{N^{\xi+(k-3)/2}}{\sqrt{\eta_*}}\right\} \\
&\quad \cap\left\{\sup_{0\le t \le T}(\rho_{1,t}\rho_{2,t})^{-1}\big|\braket{\Im G_t(z_1)A\Im G_t(z_2)A}-\Im m_t(z_1)\Im m_t(z_2)\braket{A^2}\big|\le \frac{N^\xi}{\sqrt{N\eta^*}}\right\}\\
&\quad\cap \left\{\sup_{0\le t \le T}(\rho_{1,t})^{-1/2}\big|\braket{G_t(z_1)A}\big|\le \frac{N^\xi}{N\sqrt{|\Im z_1|}}\right\}\Bigg],
\end{split}
\end{equation}
where \(\eta_*:=\min\set[\big]{|\Im z_i|\given i\in[k]}\), $\rho_{i,t}:=|\Im m_t(z_i)|$, and $\rho_t^*:=\max_i\rho_{i,t}$. The fact that \(\widehat{\Omega}\) is a very high probability set follows by~\cite[Theorem 2.6]{2012.13215} for \(k=1\), by~\cite[Eq. \cred{(3.10)}]{2012.13215} for \(k=2\), and by Proposition~\ref{pro:llaw} for \(k\ge 3\). In particular, since \(\Im m_t(z_1)\Im m_t(z_2)\braket{A^2}\) is bounded by $\rho_{1,t}\rho_{2,t}$ for \(k=2\), we have
\[
\sup_{0\le t\le T}\sup_{z_1, z_2}(\rho_{1,t}\rho_{2,t})^{-1} \braket{\Im G_t(z_1) A\Im G_t(z_2) A}\lesssim 1,
\] 
on the very high probability event \(\widehat{\Omega}_{\omega,\xi}\),
which, by spectral theorem, implies 
\begin{equation}\label{eq:apriori}
\sup_{0\le t\le T}\max_{i,j \in[N]} |\braket{{\bm u}_i(t), A {\bm u}_j(t)}|\le N^{-1/2+\omega} \qquad \,\,
 \mbox{on \(\;\widehat{\Omega}_{\omega,\xi}\cap \Omega_\xi\)}.
\end{equation}

\begin{proposition}\label{prop:mainimprov}
For any scale satisfying \(N^{-1}\ll \eta\ll T_1\ll \ell N^{-1}\ll K N^{-1}\), and any small \(\epsilon, \xi>0\) it holds
\begin{equation}\label{eq:l2b}
\norm{h_{T_1}(\cdot; \ell, K, {\bm y})}_2\lesssim K^{n/2}\mathcal{E},
\end{equation}
with
\begin{equation}\label{eq:basimpr}
\mathcal{E}:= N^{n\xi}\left(\frac{N^\epsilon\ell}{K}+\frac{NT_1}{\ell}+\frac{N\eta}{\ell}+\frac{N^\epsilon}{\sqrt{N\eta}}+\frac{1}{\sqrt{K}}\right),
\end{equation}
uniformly for particle configuration \({\bm y}\in \Lambda^n\) supported on \(\mathcal{J}\) and eigenvalue trajectory \({\bm \lambda}\) on the high probability event \(\Omega_\xi \cap \widehat{\Omega}_{\omega,\xi}\).
\end{proposition}
\begin{proof}
Before presenting the formal proof, we explain the main idea. 
 In the sense of Dirichlet forms, we will replace the generator \(\mathcal{S}(t)\)~\eqref{g-1}--\eqref{g-2}, which
is the \emph{sum} of one-dimensional generators, with the generator \(\mathcal{A}(t)\) that corresponds to the \emph{product} of
 such operators 
 (see~\eqref{eq:defAgen} below for its definition). Considering that  \(c_{ij}\) decays 
 proportionally  with \(|i-j|^{-2}\) (using rigidity in~\eqref{eq:defc}), it is the kernel of the discrete approximation
 of the one dimensional operator \(|p|=\sqrt{-\Delta}\) on \(\R\) but lifted to the \(n\)-dimensional space \(\Lambda^{n}\). Therefore 
 one may think of \(\mathcal{L}(t)\), and its short range approximation
 \(\mathcal{S}(t)\), as a discrete analogue of \(|p_1|+|p_2|+ \cdots + |p_n|\), i.e.
 the sum of \(|p|\)-operators along all the \(n\) coordinate directions. 
 As explained in the introduction, using the short distance  regularisation of the underlying lattice,
 we really have \(\eta^{-1}[1-e^{-\eta|p|}]\) instead of \(|p|\) and the operator inequality~\eqref{replace} holds.
 The left hand side of~\eqref{replace}
 corresponds to the positive operator \((-\mathcal{A})\) and
 the right hand side corresponds to \((-\mathcal{S})\). The key Lemma~\ref{lem:replacement} below asserts that
 \(0\le (-\mathcal{A})\le C(n)(-\mathcal{S})\) in the sense of quadratic forms.
 The main purpose of this replacement is that \(\mathcal{A}\) averages independently 
 in every direction, therefore \(\mathcal{A}\) acting on the function  \(g=f\circ \phi\)
 has the effect that it averages in \emph{all} the \(i_1, i_2, \ldots\) indices in the definition of \(P(G)\),~\eqref{eq:defpg}.
 These averages yield traces of products \(\Im G A \Im G A \ldots \Im G A\) for which we have
 a good local law on the set \(\widehat\Omega\). 
 
 We now  explain the origin of the errors in~\eqref{eq:basimpr}.
 The error in the multi-\(G\) local laws give the crucial fourth error \(1/\sqrt{N\eta}\)
 in~\eqref{eq:basimpr}. The other errors come from various approximations:
 the dynamics commutes with the localization 
 up to an error of order \(\ell/K\) by Lemma~\ref{lem:exchav}, the short range cutoff dynamics
 approximates the original one up to time \(T_1\) with  an error of order \(NT_1/\ell\),
 while the removed long range part contributes with an error or order \(N\eta/\ell\) to the Dirichlet form.
  The last \(1/\sqrt{K}\) error term is technical; we do the analysis for typical index configurations
  where no two indices coincide and the coinciding indices have a volume factor of order \(1/\sqrt{K}\)
  smaller than the total volume.

 Now we start with the actual proof.  All the estimates in this proof hold uniformly for \({\bm y}\in \Lambda^n\) supported on \(\mathcal{J}\), hence from now on we fix a particle configuration \({\bm y}\). To make the presentation clearer we drop the parameters \({\bm y}, K, \ell\) 
and use the short-hand notations \(h_t({\bm x})=h_t({\bm x}; \ell, K, {\bm y})\), \(\Av=\Av(K,{\bm y})\), \(\Av({\bm x})=\Av({\bm x}; K,{\bm y})\), etc. Moreover, for any \({\bm i}, {\bm j}\in [N]^n\) by \(\sum_{{\bm i}}^*\) or \(\sum_{{\bm i}{\bm j}}^*\) we denote  the summations over indices that are all distinct, i.e.\ the \(i_1,\dots, i_n\), in the first sum, and \(i_1,\dots, i_n\), \(j_1,\dots, j_n\), in the second sum are all different. The same convention holds for summations over \({\bm a}, {\bm b}\in [2n]^n\).

Let
\begin{equation}
a_{ij}=a_{ij}(t):=\frac{\eta}{N((\lambda_i(t)-\lambda_j(t))^2+\eta^2)},
\end{equation}
and define their short range version \(a_{ij}^\mathcal{S}\) as in~\eqref{eq:ccutoff}. Define the operator \(\mathcal{A}=\mathcal{A}(t)\) by
\begin{equation}\label{eq:defAgen}
\mathcal{A}(t):=\sum_{{\bm i}, {\bm j}\in [N]^n}^*\mathcal{A}_{ {\bm i}{\bm j}}(t), \quad \mathcal{A}_{{\bm i}{\bm j}}(t)h({\bm x}):=\frac{1}{\eta}\left(\prod_{r=1}^n a_{i_r,j_r}^\mathcal{S}(t)\right)\sum_{{\bm a}, {\bm b}\in [2n]^n}^*(h({\bm x}_{{\bm a}{\bm b}}^{{\bm i}{\bm j}})-h({\bm x})),
\end{equation}
where
\begin{equation}
\label{eq:lotsofjumpsop}
{\bm x}_{{\bm a}{\bm b}}^{{\bm i}{\bm j}}:={\bm x}+\left(\prod_{r=1}^n \delta_{x_{a_r}i_r}\delta_{x_{b_r}i_r}\right)\sum_{r=1}^n (j_r-i_r) ({\bm e}_{a_r}+{\bm e}_{b_r}).
\end{equation}
We now explain the difference between the jump operator~\eqref{eq:lotsofjumpsop} and the one defined in~\eqref{eq:jumpop}. The operator~\eqref{eq:jumpop} changes two entries of ${\bm x}$ per time, instead ${\bm x}_{{\bm a}{\bm b}}^{{\bm i}{\bm j}}$ changes all the coordinates of ${\bm x}$ at the same time, i.e. let ${\bm i}:=(i_1,\dots, i_n), {\bm j}:=(j_1,\dots, j_n)\in [N]^n$, with $\{i_1,\dots,i_n\}\cap\{j_1,\dots, j_n\}=\emptyset$, then ${\bm x}_{{\bm a}{\bm b}}^{{\bm i}{\bm j}}\ne {\bm x}$ iff for all $r\in [n]$ it holds that $x_{a_r}=x_{b_r}=i_r$, see e.g. 
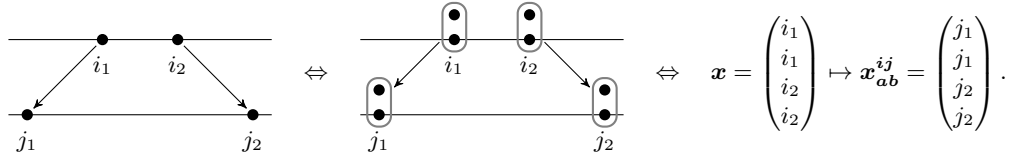
\[ \begin{tikzpicture}
  [every node/.style={fill, circle, inner sep = 1.5pt},baseline=-15pt,>=stealth']
\draw (-1.25,0) -- (2.25,0);
\draw (-1.25,-1) -- (2.25,-1);
\node[label=below:$i_1$] (i1) at (0,0) {};
\node[label=below:$i_2$] (i2) at (1,0) {};
\node[label=below:$j_1$] (j1) at (-1,-1) {};
\node[label=below:$j_2$] (j2) at (2,-1) {};
\draw [shorten >=1pt,shorten <=1pt,->] (i1) -- (j1);
\draw [shorten >=1pt,shorten <=1pt,->] (i2) -- (j2);
\end{tikzpicture}\quad \Leftrightarrow \quad 
\begin{tikzpicture}
  [every node/.style={fill, circle, inner sep = 1.5pt},baseline=-15pt,>=stealth']
\draw (-1.25,0) -- (2.25,0);
\draw (-1.25,-1) -- (2.25,-1);
\node[label=below:$i_1$] (i1) at (0,0) {};
\node[label=below:$i_2$] (i2) at (1,0) {};
\node[label=below:$j_1$] (j1) at (-1,-1) {};
\node[label=below:$j_2$] (j2) at (2,-1) {};
\node (i1p) at ($(i1)+(0,.33)$) {};
\node (i2p) at ($(i2)+(0,.33)$) {};
\node (j1p) at ($(j1)+(0,.33)$) {};
\node (j2p) at ($(j2)+(0,.33)$) {};
\draw[thick,gray,rounded corners]  ($(i1p.north west)+(-0.1,0.1)$) rectangle ($(i1.south east)+(0.1,-0.1)$);
\draw[thick,gray,rounded corners]  ($(j1p.north west)+(-0.1,0.1)$) rectangle ($(j1.south east)+(0.1,-0.1)$);
\draw[thick,gray,rounded corners]  ($(i2p.north west)+(-0.1,0.1)$) rectangle ($(i2.south east)+(0.1,-0.1)$);
\draw[thick,gray,rounded corners]  ($(j2p.north west)+(-0.1,0.1)$) rectangle ($(j2.south east)+(0.1,-0.1)$);
\draw [shorten >=8pt,shorten <=8pt,->] ($(i1)!0.5!(i1p)$) -- ($(j1)!0.5!(j1p)$);
\draw [shorten >=8pt,shorten <=8pt,->] ($(i2)!0.5!(i2p)$) -- ($(j2)!0.5!(j2p)$);
\end{tikzpicture}\quad \Leftrightarrow\quad
\bm x=\begin{pmatrix}i_1\\ i_1\\ i_2\\ i_2\end{pmatrix}\mapsto  \bm x_{\bm a\bm b}^{\bm i\bm j}=\begin{pmatrix}j_1\\ j_1\\ j_2\\ j_2\end{pmatrix}.
\]

Note that \(\mu({\bm x})\equiv 1\) on \(\Lambda^n\) is a reversible measure for the generator \(\mathcal{A}(t)\) (as a consequence of \(({\bm x}_{{\bm a}{\bm b}}^{{\bm i}{\bm j}})_{{\bm a}{\bm b}}^{{\bm j}{\bm i}}={\bm x}\) for any fixed \({\bm a}, {\bm b}\) and for any \({\bm x}\) such that \({\bm x}_{{\bm a}{\bm b}}^{{\bm i}{\bm j}}\ne {\bm x}\)), and that \(\pi({\bm x})\sim C(n) \mu({\bm x})\) for all \({\bm x}\in\Lambda^n\) (see~\eqref{eq:boundsrevmeasure}). We define the scalar product with respect to the measure \(\mu({\bm x})\) analogously to~\eqref{eq:scalpro}, and we denote it by \(\braket{\cdot,\cdot}_{\Lambda^n,\mu}\).

We now analyse the time evolution  of \( \norm{h_t}_2^2\):
\begin{equation}\label{eq:l2der}
\partial_t \norm{h_t}_2^2=2\braket{h_t, \mathcal{S}(t) h_t}_{\Lambda^n}.
\end{equation}

The main ingredient to give an upper bound on~\eqref{eq:l2der} is the following lemma, whose proof is postponed at the end of this section.

\begin{lemma}\label{lem:replacement}
Let \(\mathcal{S}(t)\), \(\mathcal{A}(t)\) be the generators defined in~\eqref{g-2} and~\eqref{eq:defAgen}, respectively. Then there exists a constant \(C(n)>0\), which depends only on \(n\), such that
\begin{equation}\label{eq:fundbound}
\braket{h, \mathcal{S}(t) h}_{\Lambda^n, \pi}\le C(n) \braket{h, \mathcal{A}(t) h}_{\Lambda^n,\mu}\le 0,
\end{equation}
for any \(h\in L^2(\Lambda^n)\), on the very high probability set \( \Omega_\xi\cap\widehat{\Omega}_{\xi,\omega}\).
\end{lemma}

From now on by \(C(n)\) we denote a constant that depends only on \(n\) and that may change from line to line.

Next, combining~\eqref{eq:l2der}--\eqref{eq:fundbound}, and using that \({\bm x}_{{\bm a}{\bm b}}^{{\bm i}{\bm j}}={\bm x}\) unless \({\bm x}_{a_r}={\bm x}_{b_r}=i_r\) for all \(r\in [n]\), we conclude that
\begin{equation}\label{eq:boundderl2}
\begin{split}
\partial_t \norm{h_t}_2^2&\le C(n) \braket{h_t,\mathcal{A}(t) h_t}_{\Lambda^n,\mu} \\
&=\frac{C(n) }{2\eta}\sum_{{\bm x}\in \Lambda^n}\sum_{{\bm i},{\bm j}\in [N]^n}^* \left(\prod_{r=1}^n a_{i_r j_r}^\mathcal{S}(t)\right)\sum_{{\bm a}, {\bm b}\in [2n]^n}^*\overline{h_t}({\bm x})\big(h_t({\bm x}_{{\bm a}{\bm b}}^{{\bm i} {\bm j}})-h_t({\bm x})\big)\Psi({\bm x}),
\end{split}
\end{equation}
where for any fixed \({\bm i},{\bm a},{\bm b}\) we defined
\[
\Psi(\bm x)=\Psi_{{\bm i},{\bm a},{\bm b}}(\bm x):=\left(\prod_{r=1}^n \delta_{x_{a_r}i_r}\delta_{x_{b_r}i_r}\right).
\]

Define 
\begin{equation}
\Gamma:=\{ {\bm x}\in\Lambda^n : \, d({\bm x},{\bm y})\le 3 K\} \subset\Lambda^n,
\end{equation}
and note that by the finite speed of propagation estimate in Proposition~\ref{pro:finitespeed}
and the support of \(h_0({\bm x})\),  the function \(h_t({\bm x})\) is supported on \(\Gamma\) 
up to an  exponentially small error term (see~\cite[Eqs. (5.76)-(5.77)]{2005.08425} for 
a more detailed calculation). For simplicity, 
 for the rest of the proof we treat  \(h_t({\bm x})\) as if it were supported on \(\Gamma\), neglecting the 
 exponentially small error term of size \(\pi(\Lambda^n\setminus \Gamma)e^{-N^\epsilon}\le N^{2n}e^{-N^\epsilon}\).
 Since the dynamics is a linear contraction in \(L^\infty\), this small error term remains small throughout the 
 whole evolution.

Now we consider the term with \(|h_t({\bm x})|^2\) in~\eqref{eq:boundderl2} (here we use the notation \(\Psi({\bm x})=\Psi_{{\bm i},{\bm a},{\bm b}}({\bm x})\)):
\begin{equation}\label{eq:neweq}
\begin{split}
&-\sum_{{\bm x}\in\Gamma}|h_t({\bm x})|^2\sum_{{\bm a},{\bm b}\in [2n]^n}^*\sum_{{\bm i}, {\bm j}}^*\Psi({\bm x}) \left(\prod_{r=1}^n a_{i_r j_r}^\mathcal{S}(t)\right) \\
&\qquad=-\sum_{{\bm x}\in\Gamma}|h_t({\bm x})|^2\sum_{{\bm a},{\bm b}\in [2n]^n}^*\sum_{{\bm i}}^*\Psi({\bm x})\prod_{r=1}^n\left(\sum_{j_r} a_{i_r j_r}^\mathcal{S}(t)+\mathcal{O}\left(\frac{1}{N\eta}\right)\right) \\
&\qquad=-\sum_{{\bm x}\in\Gamma}|h_t({\bm x})|^2\sum_{{\bm a},{\bm b}\in [2n]^n}^*\sum_{{\bm i}}^*\Psi({\bm x})\prod_{r=1}^n\left( \sum_{j_r} a_{i_r j_r}(t)+\mathcal{O}\left(\frac{1}{N\eta}+\frac{N\eta}{\ell}\right)\right)\\
&\qquad\le - C(n)\sum_{{\bm x}\in\Gamma}|h_t({\bm x})|^2\sum_{{\bm a},{\bm b}\in [2n]^n}^*\sum_{{\bm i}}^* \Psi({\bm x})
\end{split}
\end{equation}
on the very high probability event \(\Omega_\xi\), where the error term in the second line comes from adding back the 
finitely many excluded summands \( j_r\in \{i_1,\dots, i_n\}\) and \(j_r\in \{j_1, \ldots, j_{r-1},j_{r+1}, \ldots j_n\}\).
The  new error in the third line  comes from removing the short range restriction from \(a_{i_rj_r}^{\mathcal S}\), i.e.\ adding back the regimes \(|j_r-i_r|> \ell\) using 
\begin{equation}\label{eq:longreg}
   \sum_{j_r: |j_r-i_r|> \ell } a_{i_r j_r}(t) \le \frac{N\eta}{\ell}.
\end{equation}
 Finally,   to go from the third to the fourth line in~\eqref{eq:neweq} we used the local law
\begin{equation}\label{eq:singlegllaw}
\sum_{j_r}a_{i_r j_r}(t)=\braket{\Im G_t(\lambda_{i_r}+\ii\eta)}= \Im m_t(\lambda_{i_r}+\ii\eta)+\mathcal{O}\big(N^\xi (N\eta)^{-1}\big),
\end{equation}
with very high probability on the event \(\Omega_\xi\), and that \(\Im m_t(\lambda_{i_r}+\ii\eta)\sim 1\) in the bulk
whenever \(\eta\ge N^{-1+\xi}\).

We now bound the last line in~\eqref{eq:neweq} in terms of \(-\norm{h_t}^2\) plus a small error term by removing the restriction from the \({\bm i}\)-summation:
\begin{equation}\label{eq:newbdiag}
\begin{split}
-\sum_{{\bm x}\in\Gamma}|h_t({\bm x})|^2\sum_{{\bm a},{\bm b}\in [2n]^n}^*\sum_{{\bm i}}^* \Psi({\bm x})&=-\sum_{{\bm x}\in\Gamma}|h_t({\bm x})|^2\sum_{{\bm a},{\bm b}\in [2n]^n}^*\sum_{{\bm i}} \Psi({\bm x}) \\
&\quad +\sum_{{\bm x}\in\Gamma}|h_t({\bm x})|^2\sum_{{\bm a},{\bm b}\in [2n]^n}^*\left(\sum_{{\bm i}}-\sum_{{\bm i}}^*\right)\Psi({\bm x}) \\
&\le - C(n) \norm{h_t}_2^2+C(n)N^\xi K^{n-1}.
\end{split}
\end{equation}
To estimate the first term in the right hand side we used that
\(\sum_{{\bm a}{\bm b}}^*\sum_{{\bm i}} \Psi_{{\bm i},{\bm a},{\bm b}}({\bm x})\ge 1\), \(1\ge C(n) \pi({\bm x})\) for all \({\bm x}\in \Gamma\), and that we can add back the regime \(\Lambda^n\setminus \Gamma\) at the price of a negligible \(N^n e^{-N^\epsilon}\) error term, by finite speed of propagation.
For the second term in the right hand side of~\eqref{eq:newbdiag}
we estimated \(\norm{h_t}_\infty\le N^\xi\)
as a consequence of \(\norm{h_0}_\infty\le N^\xi\) and the fact that the evolution is an \(L^\infty\)-contraction. Finally we used  the fact that
\[
\sum_{{\bm a},{\bm b}\in [2n]^n}^*\left(\sum_{{\bm i}}-\sum_{{\bm i}}^*\right)\Psi_{{\bm i},{\bm a},{\bm b}}({\bm x})\ne 0,
\]
only if there exist \(a,b,c,d\in [2n]\), all distinct, such that \({\bm x}_a={\bm x}_b={\bm x}_c={\bm x}_d\). The  volume of this one codimensional subset of \(\Gamma\) is \(C(n) K^{n-1}\), i.e.\ by  factor \(K^{-1}\) smaller than the volume of \(\Gamma\) which is of order \(K^n\).

Finally, combining~\eqref{eq:neweq} and~\eqref{eq:newbdiag}, we conclude the estimate for the term containing \(|h_t({\bm x})|^2\) in~\eqref{eq:boundderl2}:
\begin{equation}\label{eq:addb}
 -\sum_{{\bm x}\in\Gamma}|h_t({\bm x})|^2\sum_{{\bm a},{\bm b}\in [2n]^n}^*\sum_{{\bm i}, {\bm j}}^* \Psi({\bm x})\left(\prod_{i=1}^r a_{i_r j_r}^\mathcal{S}(t)\right)\le -C_1(n)\norm{h_t}_2^2+C(n) N^\xi K^{n-1}.
\end{equation}

Then, using~\eqref{eq:boundderl2} together with~\eqref{eq:addb}, we conclude that
\begin{equation}\label{eq:halfwatro}
\begin{split}
&\partial_t\norm{h_t}_2^2\le -\frac{C_1(n)}{\eta}\norm{h_t}_2^2+\frac{C(n)N^\xi K^{n-1}}{\eta} \\
&\quad +\frac{C_2(n)}{\eta}\sum_{{\bm x}\in \Gamma}|h_t({\bm x})|\sum_{{\bm a}, {\bm b}\in [2n]^n}^*\sum_{{\bm i}}^*\Psi({\bm x})\Bigg|\sum_{{\bm j}}^* \left(\prod_{r=1}^n a_{i_r j_r}^\mathcal{S}(t)\right) h_t({\bm x}_{{\bm a}{\bm b}}^{{\bm i} {\bm j}})\Bigg|
\end{split}
\end{equation}
for some constants \(C_1(n), C_2(n)>0\), on the event \(\Omega_\xi\) with \(\xi>0\) arbitrarily small. In order to conclude the bound of \(\partial_t\norm{h_t}_2^2\) we are now left with the estimate of the last line in~\eqref{eq:halfwatro}.

In the remainder of the proof we will show that
\begin{equation}\label{eq:remaingoal}
\begin{split}
 &\frac{C_2(n)}{\eta}\sum_{{\bm x}\in \Gamma}|h_t({\bm x})|\sum_{{\bm a} , {\bm b}\in [2n]^n}^*\sum_{{\bm i}}^*\Psi({\bm x})\Bigg|\sum_{{\bm j}}^* \left(\prod_{r=1}^n a_{i_r j_r}^\mathcal{S}(t)\right) h_t({\bm x}_{{\bm a}{\bm b}}^{{\bm i} {\bm j}})\Bigg| \\
&\qquad\quad\le \frac{C_1(n)}{2\eta}\norm{h_t}_2^2+ \frac{C_3(n)}{\eta} \mathcal{E}^2K^n,
\end{split}
\end{equation}
with \(\mathcal{E}\) defined in~\eqref{eq:basimpr} and \(C_1(n)\) being the constant from the first line of~\eqref{eq:halfwatro}. Note that using~\eqref{eq:remaingoal} we readily conclude the proof of~\eqref{eq:l2b} by
\begin{equation}
\partial_t\norm{h_t}_2^2\le -\frac{C_1(n)}{2\eta}\norm{h_t}_2^2+\frac{C_3(n)}{\eta}\mathcal{E}^2K^n,
\end{equation}
which implies \(\norm{h_{T_1}}_2^2\le C(n) \mathcal{E}^2 K^n\), by a simple Gronwall inequality, using that \(T_1\gg \eta\).

We now conclude the proof of~\eqref{eq:l2b} proving the bound in~\eqref{eq:remaingoal}. We start with the analysis of
\begin{equation}\label{eq:interest}
\sum_{{\bm j}\in [N]^n}^* \left(\prod_{r=1}^n a_{i_r j_r}^\mathcal{S} (t)\right) h_t({\bm x}_{{\bm a}{\bm b}}^{{\bm i} {\bm j}})
\end{equation}
for any fixed \({\bm x}\in\Gamma\), \({\bm i}\in [N]^n\), \({\bm a}, {\bm b}\in [2n]^n\) with all distinct coordinates such that \(\Psi({\bm x})\ne 0\). 
It will be very important that the configuration \(\phi( {\bm x}_{{\bm a}{\bm b}}^{{\bm i} {\bm j}} )\) contains
exactly one particle at every index \(j_r\), i.e.\ we have
\begin{equation}\label{one}
\prod_{l=1}^N (n_l({\bm x}_{{\bm a}{\bm b}}^{{\bm i}{\bm j}})-1)!!=1.
\end{equation}

Similarly to~\cite[Eqs. (5.89)--(5.91), Eqs. (5.95)--(5.97)]{2005.08425}, using that the function \(f({\bm x})\equiv \bm1(n \,\, \mathrm{even})\) is in the kernel of \(\mathcal{S}(t)\), for any fixed \({\bm x}\in\Gamma\), and for any fixed \({\bm i}\), \({\bm a}\), \({\bm b}\) we conclude that
\begin{equation}\label{eq:fundrelicon}
\begin{split}
&h_t({\bm x}_{{\bm a}{\bm b}}^{{\bm i} {\bm j}})\\
&=\mathcal{U}_\mathcal{S}(0,t)\big((\Av g_0)({\bm x}_{{\bm a}{\bm b}}^{{\bm i}{\bm j}})-(\Av\bm1(n \,\, \mathrm{even}))({\bm x}_{{\bm a}{\bm b}}^{{\bm i}{\bm j}})\big) \\
&=\Av({\bm x}_{{\bm a}{\bm b}}^{{\bm i}{\bm j}})\big(\mathcal{U}_\mathcal{S}(0,t)g_0({\bm x}_{{\bm a}{\bm b}}^{{\bm i}{\bm j}})-\bm1(n\,\, \mathrm{even})\big)+\mathcal{O}\left(\frac{N^{\epsilon+n\xi} \ell}{K}\right) \\
&=\left(\Av({\bm x})+\mathcal{O}\left(\frac{\ell}{K}\right)\right)\left(\mathcal{U}(0,t)g_0({\bm x}_{{\bm a}{\bm b}}^{{\bm i}{\bm j}})-\bm1(n\,\, \mathrm{even})+\mathcal{O}\left(\frac{N^{1+n\xi}t}{\ell}\right)\right)+\mathcal{O}\left(\frac{N^{\epsilon+n\xi} \ell}{K}\right) \\
&=\Av({\bm x})\big(g_t({\bm x}_{{\bm a}{\bm b}}^{{\bm i}{\bm j}})\big)-\bm1(n\,\, \mathrm{even})\big)+\mathcal{O}\left(\frac{N^{\epsilon+n\xi} \ell}{K}+\frac{N^{1+n\xi}t}{\ell}\right),
\end{split}
\end{equation}
where the error terms are uniform in \({\bm x}\in\Gamma\). Note that to go from the first to the second line in
~\eqref{eq:fundrelicon} we used Lemma~\ref{lem:exchav}, to go from the second to the third line we used Lemma~\ref{lem:shortlongapprox} together with the a priori bound \(\norm{g_t}_\infty \le N^{n\xi}\) for any \(0\le t\le T\)
 on the very high probability event \(\widehat{\Omega}_{\omega,\xi}\), and that
\[
|\Av({\bm x})-\Av({\bm x}_{{\bm a}{\bm b}}^{{\bm i}{\bm j}})|\le \frac{1}{K}\norm{{\bm x}-{\bm x}_{{\bm a}{\bm b}}^{{\bm i}{\bm j}}}_1\le \frac{2n\ell}{K},
\]
where \(\| {\bf x}\|_1= \sum_{c=1}^{2n} |x_c|\). To go from the third to the fourth line in~\eqref{eq:fundrelicon}
we used that \(|\Av({\bm x})|\le 1\) and again that \(\norm{g_t}_\infty \le N^{n\xi}\).
Then, from~\eqref{eq:fundrelicon}, we conclude that
\begin{equation}\label{eq:medgoal}
\begin{split}
 \sum_{{\bm j}}^* \left(\prod_{r=1}^n a_{i_r j_r}^\mathcal{S}(t)\right) h_t({\bm x}_{{\bm a}{\bm b}}^{{\bm i} {\bm j}})&=\Av({\bm x})\sum_{{\bm j}}^* \left(\prod_{r=1}^n a_{i_r j_r}^\mathcal{S}(t)\right) \big(g_t({\bm x}_{{\bm a},{\bm b}}^{{\bm i} {\bm j}})-\bm1(n\,\, \mathrm{even})\big) \\
&\quad +\mathcal{O}\left(\frac{N^{\epsilon+n\xi}\ell}{K}+\frac{N^{1+n\xi}T_1}{\ell}\right).
\end{split}
\end{equation}
From now on we will omit the \(\Av({\bm x})\) prefactor
 in~\eqref{eq:medgoal}, since \(|\Av({\bm x})|\le 1\).

Using the definition of \(g_t\) from~\eqref{eq:defg}
and~\eqref{eq:deff}, for any \({\bm x}\in\Gamma\) such that \(\Psi({\bm x})\ne 0\), and for any fixed \({\bm i}\), \({\bm a}, {\bm b}\), dropping the \(t\)-dependence of the eigenvalues \(\lambda_i=\lambda_i(t)\), we have 
\begin{equation}\label{eq:leadorcancel1}
\begin{split}
&\sum_{{\bm j}}^*\left(\prod_{r=1}^n a_{i_r j_r}^\mathcal{S}(t)\right)\big(g_t({\bm x}_{{\bm a}{\bm b}}^{{\bm i} {\bm j}})-\bm1(n\,\, \mathrm{even})\big) \\
&=\sum_{{\bm j}}^*\left(\prod_{r=1}^n a_{i_r j_r}(t)\right)\left(\frac{N^{n/2}}{\braket{A^2}^{n/2} 2^{n/2}(n-1)!!}\sum_{G\in \mathcal{G}_{{\bm \eta}^{{\bm j}}}}P(G)-\bm1(n\,\, \mathrm{even})\right)+\mathcal{O}\left(\frac{N^{1+n\xi}\eta}{\ell}\right) \\
&=\sum_{{\bm j}}\left(\prod_{r=1}^n a_{i_r j_r}(t)\right)\left(\frac{N^{n/2}}{\braket{A^2}^{n/2} 2^{n/2}(n-1)!!}\sum_{G\in \mathcal{G}_{{\bm \eta}^{{\bm j}}}} P(G)-\bm1(n\,\, \mathrm{even})\right)+\mathcal{O}\left(\frac{N^{n\xi}}{N\eta}+\frac{N^{1+n\xi}\eta}{\ell}\right).
\end{split}
\end{equation}  
Note that in~\eqref{eq:leadorcancel1} we used the notation \({\bm \eta}^{{\bm j}}:=\phi({\bm x}_{{\bm a}{\bm b}}^{{\bm i} {\bm j}})\) to denote the particle configuration which has exactly one particle at each site \(\{j_1,\dots, j_n\}\). Note that in the last line of~\eqref{eq:leadorcancel1} we do not exclude the possibility that two indices \(j\) may assume the same value, since the sum is unrestricted. In the second and third lines of~\eqref{eq:leadorcancel1} we simply omitted the conditional expectation \(\E[\cdots|{\bm \lambda}]\) to shorten the formulas. Since all subsequent 
estimates hold with  high probability, the conditional expectation does not
play a role. When going from the first to the second line of~\eqref{eq:leadorcancel1} we removed
the short range restriction, as in~\eqref{eq:longreg}, by adding back the summations over the regimes \(|j_r-i_r|> \ell\), and we also used~\eqref{one}
since the coordinates of \({\bm j}\) are all distinct, and so that 
\({\mathcal M}({\bm \eta}^{{\bm j}})=1\) in the definition of \(g_t\) in~\eqref{eq:defg} and~\eqref{eq:deff}. Additionally, the error term in the third line of~\eqref{eq:leadorcancel1} comes from adding back the missing \(j_r\)-summations; in this bound we used the a priori bound  \(|P(G)|\le N^{n\xi-n/2}\) on the very high probability event \(\widehat{\Omega}_{\omega,\xi}\) and~\eqref{eq:singlegllaw}.

  We now use the definition of \(P(G)\) in~\eqref{eq:defpg} on the right hand side of~\eqref{eq:leadorcancel1}. Since every particle is doubled we may rewrite the sum over perfect matchings as 
  \begin{equation}\label{eq PG sum}
    \sum_{G\in\cG_{\bm\eta^{\bm j}}} P(G) = \sum_{G\in \Gr_2[n]} \prod_{(v_1\cdots v_k)\in\Cyc(G)} (2k-2)!! p_{j_{v_1}j_{v_2}}\cdots p_{j_{v_k}j_{v_1}},
  \end{equation} 
  where \(\Gr_2[n]\) denotes the set of 2-regular multi-graphs (possibly with loop-edges) on \([n]\) and \(\Cyc(G)\) denoting the collection of cycles in any such graph \(G\in \Gr_2[n]\). The combinatorial factor \((2k-2)!!\) is due to the fact that for each cycle in \(G\) there are \((2k-2)!!\) equivalent perfect matchings giving the very same cyclic monomial. For example, for \(n=2\) there are two 2-regular multi-graphs, \((11),(22)\) and \((12),(12)\) and thus \(\sum_{G\in\cG_{\bm\eta^{\bm j}}}P(G)=2p_{j_1j_2}^2 + p_{j_1j_1}p_{j_2j_2}\). Similarly, for \(n=3\) there are the graphs \[\set{(11),(22),(33)},\set{(12),(12),(33)},\set{(13),(13),(22)},\set{(23),(23),(11)},\set{(12),(23),(13)}\] yielding 
  \[\sum_{G\in\cG_{\bm\eta^{\bm j}}}P(G)=p_{j_1j_1}p_{j_2j_2}p_{j_3j_3} + 2p_{j_1j_1}p_{j_2j_3}^2 + 2p_{j_2j_2}p_{j_1j_3}^2 + 2p_{j_3j_3}p_{j_1j_2}^2 + 8p_{j_1j_2}p_{j_2j_3}p_{j_1j_3}.\]
  For each graph \(G\in\Gr_2[n]\) we may use the spectral theorem to perform the \(\bm j\) summation as 
  \begin{equation}\label{j performed sum}
    \sum_{j_{v_1},\ldots,j_{v_k}} \Bigl(\prod_{r\in[k]}a_{i_{v_r}j_{v_r}}(t)\Bigr) p_{j_{v_1}j_{v_2}}\cdots p_{j_{v_k}j_{v_1}} = N^{1-k}F_k(v_1,\ldots,v_k)
  \end{equation}
  with 
  \[F_k(v_1,\ldots,v_k):=\braket{\Im G_t(\lambda_{i_{v_1}}+\ii\eta)A\cdots \Im G_t(\lambda_{i_{v_k}}+\ii\eta)A}.\]
  Since each vertex appears in exactly one cycle, we can use~\eqref{j performed sum} to perform the summation for the indices corresponding to any cycle separately and obtain
  \begin{equation}\label{eq P(G) F}
    \begin{split}
      \sum_{{\bm j}}\left(\prod_{r=1}^n a_{i_r j_r}(t)\right)\sum_{G\in \mathcal{G}_{{\bm \eta}^{{\bm j}}}} P(G) =  \sum_{E\in \Gr_2[n]} \prod_{(v_1\cdots v_k)\in\Cyc(E)} (2k-2)!!N^{1-k} F_k(v_1,\ldots, v_k).
    \end{split}
  \end{equation}
  We note that from~\eqref{eq:hatomega} for each \(k\ge 1\) we have the estimate
  \begin{equation}\label{F est}
    F_k(v_1,\ldots v_k) = \bm 1(k=2) \braket{A^2}\Im m(z_{i_{v_1}})\Im m(z_{i_{v_k}}) + \landauO*{N^\xi \frac{N^{k/2-1}}{\sqrt{N\eta}}}
  \end{equation}
  on the high-probability set \(\wh\Omega\). By using~\eqref{F est} within~\eqref{eq P(G) F} and using the fact that there are \(\bm1(n\text{ even})(n-1)!!\) graphs in \(\Gr_2[n]\) all of which cycles have length two, it follows that 
  \begin{equation}
   ~\eqref{eq P(G) F} = \bm1(n\text{ even}) (n-1)!! 2^{n/2} N^{-n/2} \braket{A^2}^{n/2} \prod_{r\in[n]}\Im m(z_{i_r}) + \landauO*{N^\xi\frac{N^{-n/2}}{\sqrt{N\eta}}}
  \end{equation}
  and from~\eqref{eq:leadorcancel1} we conclude 
  \begin{equation}\label{eq:finbound}
    \Psi({\bm x})\sum_{{\bm j}}^*\left(\prod_{r=1}^n a_{i_r j_r}^\mathcal{S}(t)\right)\big(g_t({\bm x}_{{\bm a}{\bm b}}^{{\bm i} {\bm j}})-\bm1(n\,\, \mathrm{even})\big) = \Psi({\bm x})\mathcal{O}\left(\frac{N^\xi}{\sqrt{N\eta}}+\frac{N^\xi}{N\eta}+\frac{N^{1+\xi}\eta}{\ell}\right).
  \end{equation}
  We remark that in estimating the error term we used that $\braket{A^2}\ge \delta'$.

Combining~\eqref{eq:medgoal} and~\eqref{eq:finbound}, we get that
\begin{equation}\label{eq:almthere}
\Psi({\bm x})\left|\sum_{{\bm j}}^* \left(\prod_{r=1}^n a_{i_r j_r}^\mathcal{S}(t)\right) h_t({\bm x}_{{\bm a}{\bm b}}^{{\bm i} {\bm j}})\right| \le C(n) \Psi({\bm x})\mathcal{E}
\end{equation}
and finally, by~\eqref{eq:almthere}, we conclude that
\begin{equation}\label{eq:gronw}
\mathrm{l.h.s.}~\eqref{eq:remaingoal}\le \frac{C_1(n)}{2\eta}\norm{h_t}_2^2+ \frac{C_3(n)}{\eta} \mathcal{E}^2K^n,
\end{equation}
where we used that for any fixed \({\bm x}\in\Lambda^n\) we have
\[
\sum_{{\bm a}, {\bm b}\in [2n]^n}^*\sum_{{\bm i}}^*\Psi_{{\bm i},{\bm a},{\bm b}}({\bm x})\le C(n),
\]
and that
\[
\left|\sum_{{\bm x}\in \Gamma} h_t({\bm x})\mathcal{E}\right|\le \frac{C_1(n)}{2} \sum_{{\bm x}\in \Gamma} \pi({\bm x}) |h_t({\bm x})|^2+C_3(n)\mathcal{E}^2K^n,
\]
by the Schwarz inequality,  the bound \(1\le  \pi({\bm x})\)  from~\eqref{eq:boundsrevmeasure}, and \(\sum_{{\bm x}\in \Gamma} \pi({\bm x})\le C(n) K^n\).
Note that by balancing between the two terms
in the Schwarz inequality we could achieve the same constant \(C_1(n)\)   with an additional 1/2 factor
in front of the \(\norm{h_t}_2^2\) term as in the leading term in~\eqref{eq:halfwatro} with a minus sign. This concludes the proof of the bound in~\eqref{eq:remaingoal}.
\end{proof}

\begin{proof}[Proof of Lemma~\ref{lem:replacement}]

All along the proof \(C(n)>0\) is a constant that depends only on \(n\) and that may change from line to line.

We consider
\begin{equation}\label{eq:squaform}
\begin{split}
\braket{h, \mathcal{S} (t)h}_{\Lambda^n,\pi}&=-\frac{1}{2}\sum_{{\bm x}\in \Lambda^n}\pi({\bm x})\sum_{j\ne i} c_{ij}^{\mathcal{S}}(t)\frac{n_j({\bm x})+1}{n_i({\bm x})-1}\sum_{a\ne b\in [2n]}\big|h({\bm x}_{ab}^{i j})-h({\bm x})\big|^2 \\
&\le -\frac{C(n)}{\eta}\sum_{{\bm x}\in \Lambda^n}\sum_{j\ne i}a_{ij}^{\mathcal{S}} (t)\sum_{a\ne b\in [2n]}\big|h({\bm x}_{ab}^{ij})-h({\bm x})\big|^2
\end{split}
\end{equation}
and
\begin{equation}\label{eq:aquaform}
\braket{h, \mathcal{A} (t)h}_{\Lambda^n,\mu}=-\frac{1}{2\eta}\sum_{{\bm x}\in \Lambda^n}\sum_{{\bm i}, {\bm j}}^* \left(\prod_{r=1}^n a_{i_r j_r}^\mathcal{S}(t)\right)\sum_{{\bm a},{\bm b}\in [2n]^n}^*\big|h({\bm x}_{{\bm a} {\bm b}}^{{\bm i} {\bm j}})-h({\bm x})\big|^2.
\end{equation}
Note that in~\eqref{eq:squaform} we used that \(a_{ij}^{\mathcal{S}}(t)\le \eta c_{ij}^{\mathcal{S}}(t)\) to compare the kernels, that \(\pi ({\bm x})\ge 1\) uniformly in \({\bm x}\in\Lambda^n\) and finally 
that \(n_j({\bm x})+1\ge 1\), \(1\le n_i({\bm x})-1\le n\) for \({\bm x}\) and \(i\) such that  \(h({\bm x}_{ab}^{ij})\ne h({\bm x})\).

We start with the bound
\begin{equation}\label{eq:tel}
\begin{split}
&\sum_{{\bm x}\in \Lambda^n}\sum_{{\bm i}, {\bm j}\in [N]^n}^* \left(\prod_{r=1}^n a_{i_r j_r}^\mathcal{S}(t)\bm1(n_{i_r}({\bm x})>0)\right)\sum_{{\bm a},{\bm b}\in [2n]^n}^*\big|h({\bm x}_{{\bm a}{\bm b}}^{{\bm i} {\bm j}})-h({\bm x})\big|^2 \\
&\quad\le C(n)\sum_{{\bm x}\in \Lambda^n}\sum_{{\bm i} ,{\bm j}}^* \left(\prod_{r=1}^n a_{i_r j_r}^\mathcal{S}(t)\bm1(n_{i_r}({\bm x})>0)
\right)\sum_{l=1}^n\sum_{{\bm a},{\bm b}\in [2n]^n}^*\big|h(({\bm y}_{l-1})_{a_l b_l}^{i_l j_l})-h({\bm y}_{l-1})\big|^2,
\end{split}
\end{equation}
where we recursively defined 
\({\bm y}_0={\bm x}, {\bm y}_1, {\bm y}_2 \ldots, {\bm y}_n={\bm x}_{{\bm a}{\bm b}}^{{\bm i}{\bm j}}\)
by performing  the jumps \(i_1\to j_1\), \(i_2\to j_2\), etc.,  one by one (assuming that the choice of \((a_l,b_l)\) allows it, otherwise \({\bm y}_l={\bm y}_{l-1}\):
\begin{equation}
{\bm y}_0={\bm y}_0({\bm x}):={\bm x}, \qquad  {\bm y}_{l}={\bm y}_{l}({\bm x}):=({\bm y}_{l-1})_{a_l b_l}^{i_l j_l}.
\end{equation}
In the first line of~\eqref{eq:tel} we could add the indicator \(\bm1(n_{i_r}(\bm x)>0)\) since in case \(n_{i_r}(\bm x)=0\) for some \(r\) it holds that \(\bm x_{\bm a \bm b}^{\bm i \bm j}=\bm x\). 
Note that to go from the first to the second line of~\eqref{eq:tel} we wrote a telescopic sum
\[
h({\bm x}_{{\bm a}{\bm b}}^{{\bm i}{\bm j}})-h({\bm x})=\sum_{l=1}^n \big[ h(({\bm y}_{l-1})_{a_l b_l}^{i_l j_l})-h({\bm y}_{l-1})\big],
\]
and used Schwarz inequality.

Next we consider
\begin{equation}\label{eq:finsec}
\begin{split}
&\sum_{l=1}^n\sum_{{\bm x}\in \Lambda^n}\sum_{{\bm i}, {\bm j}}^*\sum_{{\bm a},{\bm b}\in [2n]^n}^* \left(\prod_{r=1}^n a_{i_r j_r}^\mathcal{S}(t)\bm1(n_{i_r}({\bm x})>0)
\right)\big|h(({\bm y}_{l-1})_{a_l b_l}^{i_l j_l})-h({\bm y}_{l-1})\big|^2 \\
&\qquad=\sum_{l=1}^n\sum_{{\bm w}\in \Lambda^n}\sum_{{\bm i}, {\bm j}}^* \sum_{{\bm a},{\bm b}\in [2n]^n}^*\left(\prod_{r=1}^n a_{i_r j_r}^\mathcal{S}(t)\bm1(n_{i_r}({\bm z}_{l-1})>0) \right)\big|h({\bm w}_{a_l b_l}^{i_l j_l})-h({\bm w})\big|^2 \\
&\qquad\le C(n) \sum_{{\bm w}\in \Lambda^n}\sum_{l=1}^n \sum_{i_l\ne j_l}a_{i_l j_l}^{\mathcal{S}}(t)\sum_{a_l\ne b_l\in [2n]}\big|h({\bm w}_{a_l b_l}^{i_l j_l})-h({\bm w})\big|^2 \\
&\qquad\qquad\quad \times \left(\prod_{r\ne l} \sum_{i_r, j_r} a_{i_r j_r}\big[\bm1(n_{i_r}({\bm w})>0)+\bm1(n_{j_r}({\bm w})>0)\big]\right) \\
&\qquad\le C(n) \sum_{{\bm w}\in \Lambda^n} \sum_{l=1}^n\sum_{i_l\ne j_l}a_{i_l j_l}^{\mathcal{S}}(t) \sum_{a_l\ne b_l\in [2n]}\big|h({\bm w}_{a_l b_l}^{i_l j_l})-h({\bm w})\big|^2 \\
&\qquad\le C(n) \sum_{{\bm w}\in \Lambda^n}\sum_{i\ne j}a_{i j}^{\mathcal{S}}(t) \sum_{a\ne b\in [2n]}\big|h({\bm w}_{a b}^{i j})-h({\bm w})\big|^2.
\end{split}
\end{equation}
Note that to go from the first to the second line we did the change of variables \({\bm w}= {\bm y}_{l-1}({\bm x})\), we used that \(({\bm x}_{a_lb_l}^{i_lj_l})_{a_lb_l}^{j_l i_l}={\bm x}\) for any \({\bm x}\in \Lambda^n\) such that \(\prod_r \bm1(n_{i_r}({\bm x})>0)\), and we defined \({\bm z}_{l-1}=(({\bm w}_{a_{l-1} b_{l-1}}^{j_{l-1} i_{l-1}})\dots)_{a_1 b_1}^{j_1 i_1}\). Moreover, to go from the second to the third line in~\eqref{eq:finsec} we used that
\begin{equation}
\begin{split}
    \prod_{r\in[n]\setminus\set{l}}\bm1(n_{i_r}({\bm z}_{l-1})>0)&\le C(n) \left(\prod_{r=1}^{l-1}\Big[\bm1(n_{j_r}({\bm w})>0)+\bm1(n_{i_r}({\bm w})>0)\Big] \right)\\
    &\qquad\qquad\times\left(\prod_{r=l+1}^n \bm1(n_{i_r}({\bm w})>0)\right)
\end{split}
\end{equation}
for \(i_1,\dots i_n, j_1,\dots, j_n\) all distinct, which follows by \(n_{i_r}({\bm z}_{l-1})=n_{i_r}({\bm w})\) if \(r \ge l+1\) and
\[
\bm1(n_{i_r}({\bm z}_{l-1})>0)\le \bm1(n_{i_r}({\bm w})>0)+\bm1(n_{j_r}({\bm w})>0)
\]
for \(r\le l-1\). In the penultimate inequality in~\eqref{eq:finsec} we also used that
\begin{equation}
\prod_{r\ne l} \sum_{i_r j_r} a_{i_r j_r}\big[\bm1(n_{i_r}({\bm w})>0)+\bm1(n_{j_r}({\bm w})>0)\big]\le C(n),
\end{equation}
on the very high probability event \(\widehat{\Omega}\).
Combining~\eqref{eq:squaform}--\eqref{eq:aquaform},~\eqref{eq:tel} and~\eqref{eq:finsec}, we finally conclude~\eqref{eq:fundbound}.
\end{proof}

\subsection{Proof of Proposition~\ref{pro:flucque}}

Fix \(1\ll NT_1\ll \ell_1\ll K\) and \(1\ll NT_2\ll \ell_2 \ll K\), 
with \(T_1\le T_2/2\). Define the \emph{lattice generator} \(\mathcal{W}(t)\) by
\begin{equation}
\mathcal{W}(t):=\sum_{i\ne j\in [N]} \mathcal{W}_{ij}(t), \qquad \mathcal{W}_{ij}(t):=c_{ij}^{\mathcal{W}}(t)\frac{n_j({\bm x})+1}{n_i({\bm x})-1}\sum_{a\ne b\in [2n]}\big(h({\bm x}_{ab}^{ij})-h({\bm x})\big),
\end{equation}
with
\begin{equation}
c_{ij}^{\mathcal{W}}(t):=\begin{cases}
c_{ij}(t) &\mathrm{if}\, i,j\in\mathcal{J}\,\, \mathrm{and} \,  1\le |i-j|\le \ell_2 \\
\frac{N}{|i-j|^2} &\mathrm{otherwise}.
\end{cases}
\end{equation}
Denote by \(\mathcal{U}_\mathcal{W}(s,t)\) the semigroup associated to the generator \(\mathcal{W}(t)\). Note that \(\mathcal{W}(t)\) is the original generator of the Dyson eigenvector flow \(\mathcal{L}\)
from~\eqref{eq:g1dker} on short scales and in the interval \(\mathcal{J}\) well inside the bulk,
 while on large scales it has an  equidistant jump rate. In~\cite{2005.08425} this replacement 
 made up for the missing rigidity (regularity) control  of the eigenvalues outside of a local interval \(\mathcal{J}\);
 in our case its role is just to handle the somewhat different scaling of the eigenvalues near the edges. We follow the setup of~\cite{2005.08425} for convenience.

On the event \(\Omega_\xi\) the coefficients \(c_{ij}^{\mathcal{W}}(t)\) satisfy~\cite[Assumption 6.8]{2005.08425} with a rate \(v=N^{1-\xi}\), for any arbitrary small \(\xi>0\), hence all the results in~\cite[Section 6]{2005.08425} apply to the generator \(\mathcal{W}(t)\).
Most importantly, the Dirichlet form of \(\mathcal{W}(t)\) satisfies a Poincar\'e inequality and, consequently 
 we have an \(L^2\to L^\infty\) ultracontractive  decay  bound for the corresponding semigroup. 
 Their scaling properties confirm the intuition that \(\mathcal{W}(t)\) is a discrete analogue of the \(|p|= \sqrt{-\Delta}\)
 operator in \(\R^{2n}\).  In the continuous setting, standard Sobolev inequality combined with the Nash method implies that
 \begin{equation}
 \label{nash}
    \| e^{-t|p|} f \|_{L^\infty(\R^{2n})} \le \frac{C(n)}{t^{n/2}}\| f\|_{L^2(\R^{2n})}
 \end{equation}
 holds for any \(L^2\) function on \(\R^{2n}\). The same decay holds for the semigroup generated by \(\mathcal{W}(t)\)
 by~\cite[Proposition 6.29]{2005.08425} (recall that~\cite{2005.08425} uses \(n\) to denote
 the dimension of the space of \({\bm x}\)'s, we use  \(2n\)).     We remark
 that the proofs in~\cite[Section 6]{2005.08425} are designed for the more involved \emph{coloured} dynamics;  here we need
 only its  simpler \emph{colourblind} version which immediately follows from the coloured version by ignoring the colors. In particular, in our case the \emph{exchange operator} \(\mathcal{E}_{ij}\) is identically zero. While a direct proof of the colourblind version is possible and it would require less combinatorial complexity, for brevity, we  directly use the results of~\cite[Section 6]{2005.08425}.

For each \({\bm y}\) supported on \(\mathcal{J}\), let \(q_t({\bm x})=q_t({\bm x};{\bm y})\) be the solution of
\begin{equation}\label{eq:q}
\begin{cases}
q_0({\bm x})=\Av({\bm x};K,{\bm y})\big(g_0({\bm x})-\bm1(n\,\, \mathrm{even})\big) \\
\partial_tq_t({\bm x})=\mathcal{S}(t)q_t({\bm x}) &\mathrm{for}\,\, 0< t\le T_1 \\
\partial_tq_t({\bm x})=\mathcal{W}(t)q_t({\bm x}) &\mathrm{for}\,\, T_1< t\le T_2,
\end{cases} 
\end{equation}
with \(\mathcal{S}(t)\) being the short-range generator on a scale \(\ell=\ell_1\) from~\eqref{g-2}. Note that \(q_t=h_t\) for any \(0\le t\le T_1\), with \(h_t\) being the solution of~\eqref{g-1}. 

By Proposition~\ref{prop:mainimprov}, choosing \(\eta=N^{-\epsilon} T_1\), we have
\begin{equation}\label{eq:boundt1}
\sup_{{\bm y}: y_a\in \mathcal{J}}\norm{q_{T_1}(\cdot;{\bm y})}_2\lesssim N^\epsilon K^{n/2}\left(\frac{\ell_1}{K}+\frac{NT_1}{\ell_1}+\frac{1}{\sqrt{NT_1}}+\frac{1}{\sqrt{K}}\right),
\end{equation}
for any arbitrary small \(\epsilon>0\), where the supremum is over all the \({\bm y}\) supported on \(\mathcal{J}\). We recall that by the finite speed of propagation estimate in Proposition~\ref{pro:finitespeed}, together with~\cite[Eq. (7.12)]{2005.08425}, the function \(q_t\) is supported on the subset of \(\Gamma\subset\Lambda^n\) such that \(d({\bm x},{\bm y})\le 3K\) for any \({\bm x}\in\Gamma\) (modulo a negligible exponentially small error term). Then, using the ultracontractivity bound for the dynamics of \(\mathcal{W}(t)\) from~\cite[Proposition 6.29]{2005.08425}, with \(v=N^{1-\xi}\), we get that
\begin{equation}\label{eq:boundt2}
\begin{split}
\sup_{{\bm y}}\norm{q_{T_2}(\cdot;{\bm y})}_\infty&=\sup_{{\bm y}}\norm{\mathcal{U}_\mathcal{W}(T_1,T_2)q_{T_1}(\cdot;{\bm y})}_\infty\\
&\le \sup_{{\bm y}}\norm{(1-\Pi)\mathcal{U}_\mathcal{W}(T_1,T_2)q_{T_1}(\cdot;{\bm y})}_\infty+\sup_{{\bm y}}\norm{\Pi\mathcal{U}_\mathcal{W}(T_1,T_2)q_{T_1}(\cdot;{\bm y})}_\infty \\
&\lesssim \frac{N^{n\xi}}{[N(T_2-T_1)]^{n/2}}\sup_{{\bm y}}\norm{q_{T_1}(\cdot;{\bm y})}_2+\frac{K^n}{N^{n (1-\xi)}},
\end{split}
\end{equation}
where \(\Pi\) is the orthogonal projection into the kernel of \(\mathcal{L}\), \(\ker(\mathcal{L})=\bigcap_{i\ne j}\ker(\mathcal{L}_{ij})\), defined in~\cite[Lemma 4.17]{2005.08425}. Note that in~\eqref{eq:boundt2} we used that by~\cite[Corollary 4.20]{2005.08425} it holds
\[
\norm{\Pi q_{T_2}}_{\infty}\le C(n) N^{-n} \norm{q_{T_2}}_1\le K^n N^{-n+n\xi},
\]
since \(\norm{q_{T_2}}_\infty\le N^{n\xi}\) on the very high probability set \(\widehat{\Omega}\). We remark that in~\cite[Proposition 6.29]{2005.08425} \(\mathcal{U}_\mathcal{W}\) is replaced by \(\mathcal{U}\), but this does not play any role since the only assumption on \(\mathcal{L}_{ij}\) used in~\cite[Section 6]{2005.08425} is that \(c_{ij}(t)\ge N^{1-\xi}|i-j|^{-2}\) (see~\cite[Definition 6.8]{2005.08425}). Combining~\eqref{eq:boundt1}--\eqref{eq:boundt2} we conclude
\begin{equation}\label{eq:boundcomb}
\sup_{{\bm y}}\norm{q_{T_2}(\cdot;{\bm y})}_\infty \lesssim N^{2\epsilon} \left(\frac{K}{NT_2}\right)^{n/2}\left(\frac{\ell_1}{K}+\frac{NT_1}{\ell_1}+\frac{1}{\sqrt{NT_1}}+\frac{1}{\sqrt{K}}\right),
\end{equation}
where we used that \(T_1\le T_2/2\).

Now we compare the solution \(q_t\) from~\eqref{eq:q} with the original dynamics \(g_t\) from~\eqref{eq:g1deq}.
This is done, after several steps, using
~\cite[Proposition 7.2]{2005.08425} with \(F_t({\bm y};{\bm y})\) replaced by \(\bm1(n\,\, \mathrm{even})\), asserting that
\begin{equation}\label{eq:approxincon}
\sup_{{\bm y}}\big|q_{T_2}({\bm y}; {\bm y})-(g_{T_2}({\bm y})-\bm1(n\,\, \mathrm{even}))\big|\lesssim N^\epsilon\left(\frac{\ell_1}{K}+\frac{NT_1}{\ell_1}+\frac{\ell_2}{K}+\frac{NT_2}{\ell_2}\right).
\end{equation}
In particular, the only thing used about \(F_t({\bm y};{\bm y})\) in the proof of~\cite[Proposition 7.2]{2005.08425} is that \(F_t\) is in the kernel of all  \(\mathcal{L}_{ij}\), and this is clearly the case for \(\bm1(n\,\, \mathrm{even})\) as well. The origins of the error terms in~\eqref{eq:approxincon} are as follows. The smooth cutoff given by the \(\mbox{Av}\) localising operator in the initial condition~\eqref{eq:q} commutes with  the
time evolution generated by \(\mathcal{S}\) up an error of order \(\ell_1/K\), see Lemma~\ref{lem:exchav}.
The difference between the original dynamics and the short range dynamics 
in the time interval \(t\in [0, T_1]\) yields the error \(NT_1/\ell_1\),
see Lemma~\ref{lem:shortlongapprox}. Similar errors hold for the approximation 
of the original dynamics by the  time evolution generated by \(\mathcal{W}\)  on the time interval \(t\in [T_1, T_2]\),
giving rise to the errors \(\ell_2/K\) and \(N(T_2-T_1)/\ell_2\le NT_2/\ell_2\).

Combining~\eqref{eq:boundcomb}--\eqref{eq:approxincon}, we conclude that
\begin{equation}\label{eq:finbfin}
\begin{split}
&\sup_{{\bm y}}\big|g_{T_2}({\bm y})-\bm1(n\,\, \mathrm{even})\big| \\
&\lesssim N^{2\epsilon}\left(\frac{\ell_1}{K}+\frac{NT_1}{\ell_1}+\frac{\ell_2}{K}+\frac{NT_2}{\ell_2}+\left(\frac{K}{NT_2}\right)^{n/2}\left(\frac{\ell_1}{K}+\frac{NT_1}{\ell_1}+\frac{1}{\sqrt{NT_1}}+\frac{1}{\sqrt{K}}\right) \right) \\
&\lesssim N^{-c/(20 n)},
\end{split}
\end{equation}
on the with very high probability event \(\Omega_\xi\cap\widehat\Omega_{\xi, \epsilon}\) 
with choosing a very small \(\xi\). 
In the last step we  optimised the error terms in the second line of~\eqref{eq:finbfin} with the choice of
\[
K=N^c, \qquad  T_2=N^{-1-c/(10 n)} K, \qquad \ell_2=\sqrt{NKT_2}, \qquad \ell_1=\sqrt{NKT_1}, \qquad T_1=\frac{\sqrt{K}}{N},
\]
with some small fixed \(0< c\le 1/2\). Finally, using that
\[
\sup_{{\bm y}: y_a\in\mathcal{J}}\big|g_{T_2}({\bm y})-\bm1(n\,\, \mathrm{even})\big|=\sup_{{\bm \eta}}\big|f_{T_2}({\bm\eta})-\bm1(n\,\, \mathrm{even})\big|
\]
by~\eqref{eq:defg}, where the supremum in the right hand side is taken over configurations \({\bm \eta}\) such that \(\eta_i=0\) for \(i\in [\delta N, (1-\delta) N]^c\) and \(\sum_i \eta_i=n\). The bound in~\eqref{eq:finbfin} concludes the proof of Proposition~\ref{pro:flucque}.

\section{Local law bounds}\label{sec:llaw}
In this section we prove the local laws needed to estimate the probability of the event \(\wh\Omega\) in~\eqref{eq:hatomega}. We recall~\cite{MR2871147} that the resolvent \(G=(W-z)^{-1}\) of the Wigner matrix \(W\) is approximately equal, 
\begin{equation}\label{local law}
  G_{ab}=\delta_{ab}m+\landauO*{\frac{N^{\xi}}{\sqrt{N\Im z}}},\quad \braket{G}=m+\landauO*{\frac{N^{\xi}}{N\Im z}}
\end{equation} to the Stieltjes transform \(m=m_\mathrm{sc}(z)\) of the semicircular distribution \(\rho_\mathrm{sc}=\sqrt{4-x^2}/2\pi\) which solves the equation 
\begin{equation}\label{mde}
  -\frac{1}{m}=m+z.
\end{equation}
\begin{proposition}\label{pro:llaw}
  Let \(k\ge 3\) and \(z_1,\ldots,z_k\in\C\setminus\R\) with \(N\min_i(\rho_i\eta_i)\ge N^\epsilon\) for some \(\epsilon>0\) with \(\eta_i:=\abs{\Im z_i}\) and \(\rho_i:=\rho(z_i)\), \(\rho(z):=\abs{\Im m(z)}/\pi\). Then for arbitrary traceless matrices \(A_1,\ldots,A_k\) with \(\norm{A_i}\lesssim 1\) we have 
  \begin{equation}
  \label{eq:llawbound}
    \abs*{\braket*{G_1A_1\ldots G_k A_k}}\lesssim N^{\xi+(k-3)/2}\sqrt{\frac{\rho^\ast}{\eta_\ast}},
  \end{equation}
  with very high probability for any \(\xi>0\), where \(\rho^\ast:=\max_i \rho_i\) and \(\eta_\ast:=\min\eta_i\). 
\end{proposition}
\begin{proof}
Using \(WG-zG=I\) and~\eqref{mde} we write
 \begin{equation}\label{G m eq}
  G=m-m\un{WG}+m\braket{G-m}G
\end{equation}
where 
\[
  \un{WG} = WG + \braket{G}G
\]
denotes a renormalization of \(WG\). More generally, for functions \(f(W)\) we define 
\[
  \un{Wf(W)} := Wf(W)- \widetilde \E \widetilde W (\partial_{\widetilde W}f)(W)
\]
with \(\partial_{\wt W}\) denoting the directional derivative in direction \(\wt W\) and \(\wt W\) being an independent GUE-matrix with expectation \(\wt\E\). We now use~\eqref{G m eq} and~\eqref{local law} for \(G_1=G(z_1)\) and \(m_1=m(z_1)\) to obtain 
\begin{equation}
  \begin{split}
    \Bigl(1-\landauO[\Big]{\frac{N^\xi}{N\eta_\ast}}\Bigr)\braket*{\prod_{i=1}^k (G_i A_i)} = m_1 \braket*{A_1 \prod_{i=2}^k (G_i A_i)} - m_1 \braket*{\un{WG_1}A_1 \prod_{i=2}^k (G_i A_i)}.
  \end{split}
\end{equation}
Together with
\[ \begin{split}
  \braket*{\un{W\prod_{i=1}^k (G_i A_i)}} &= \braket*{W\prod_{i=1}^k (G_i A_i)} + \sum_{j=1}^k\wt\E \braket*{\wt W \biggl[\prod_{i=1}^{j-1} (G_iA_i)\biggr] G_{j} \wt W \prod_{i=j}^k (G_i A_i)}\\
  &= \braket*{\un{WG_1}A_1 \prod_{i=2}^k (G_i A_i)} + \sum_{j=2}^k\braket*{\biggl[\prod_{i=1}^{j-1} (G_iA_i)\biggr] G_{j}}\braket*{ \prod_{i=j}^k (G_i A_i)}
\end{split} \]
we thus have 
  \begin{equation}\label{GA eq}
    \begin{split}
      \Bigl(1-\landauO[\Big]{\frac{N^\xi}{N\eta_\ast}}\Bigr)\braket*{\prod_{i=1}^k (G_i A_i)} &= m_1 \braket*{A_1\prod_{i=2}^k (G_i A_i)} - m_1 \braket*{\un{W\prod_{i=1}^k (G_i A_i)}} \\
      &\quad + \sum_{j=2}^k\braket*{\biggl[\prod_{i=1}^{j-1} (G_iA_i)\biggr] G_{j}}\braket*{ \prod_{i=j}^k (G_i A_i)}.
    \end{split}
  \end{equation}
  We now apply the inequality~\cite[Eq.~(5.35)]{2012.13215}
  \[ \abs{\braket{XY}} \le \Bigl[\braket{X^\ast X(YY^\ast)^{1/2}}\braket{(Y^\ast Y)^{1/2}}\Bigr]^{1/2} \]
  for arbitrary matrices \(X,Y\) to \(X=\prod_{i=1}^{j-1}(G_i A_i),Y=G_j\) to obtain  
  \begin{equation*}
    \begin{split}
      \abs*{\braket*{\prod_{i=1}^{j-1}(G_i A_i)G_{j} }} &\le \eta_1^{-1/2}\Bigl(\braket*{ A_{j-1}^\ast G_{j-1}^\ast \cdots A_1^\ast \Im G_1 A_1 \cdots G_{j-1} A_{j-1} \abs{G_j}  } \braket{ \abs{G_j}}\Bigr)^{1/2} 
    \end{split}
  \end{equation*}
  from \(G^\ast G=(\Im G)/\eta\). By spectral decomposition we may further estimate with very high probability for any \(\xi>0\)
  \[
  \begin{split}
    &\abs*{\braket*{ A_{j-1}^\ast G_{j-1}^\ast \cdots A_1^\ast \Im G_1 A_1 \cdots G_{j-1} A_{j-1} \abs{G_j}  }} \\
    &= \abs*{\frac{1}{N} \sum_{\bm a} \frac{\braket{\bm u_{a_{j}},A_{j-1}^\ast \bm u_{a_{1-j}}} \cdots \braket{\bm u_{a_{-2}},A_{1}^\ast \bm u_{a_1}}\braket{\bm u_{a_{1}},A_{1} \bm u_{a_{2}}}\cdots \braket{\bm u_{a_{j-1}},A_{j-1} \bm u_{a_{j}}} }{\bigl[(\lambda_{a_{j-1}}-z_{j-1})(\lambda_{a_{1-j}}-\ov{z_{j-1}})\cdots (\lambda_{a_2}-z_2)(\lambda_{a_{-2}}-\ov{z_2})\bigr]\abs{\lambda_{a_j}-z_j}}\Im \frac{1}{\lambda_{a_1}-z_1}}\\
    &\lesssim N^{\xi+j-2} \rho(z_1)
  \end{split},\]
  from the overlap bound \(\abs{\braket{\bm u_a, A \bm u_b}}\lesssim N^{\xi-1/2}\), and where \(\sum_{\bm a}\) is the summation over the \(2j-2\) indices \(a_1,a_{\pm 2},\ldots,a_{\pm(j-1)},a_j\) and conclude 
  \begin{equation}
    \abs*{\braket*{\prod_{i=1}^{j-1}(G_i A_i)G_{j} }} \lesssim N^{\xi+j/2-1}\frac{\sqrt{\rho_1}}{\sqrt{\eta_1}}.
  \end{equation}
  Similarly we also have
  \begin{equation*}
    \abs*{\braket*{ \prod_{i=j+1}^{k} (G_i A_i)}} \lesssim N^{\xi+(k-j)/2-1},
  \end{equation*}
  and the claim follows from~\eqref{GA eq} and the bound
  \begin{equation}
    \abs*{\braket*{\un{W\prod_{i=1}^k (G_i A_i)}}} \lesssim N^\xi N^{(k-3)/2}\frac{\sqrt{\rho^\ast}}{\sqrt{\eta_\ast}}
  \end{equation}
  on the underlined term in~\cite[Theorem 4.1, Remark 4.3]{2012.13215}.
\end{proof}

\appendix

\section{Green function comparison}\label{app:GFT}
Here we briefly recall the standard Green function comparison method for eigenvector statistics. The only novelty is that in addition to the standard entry-wise local law, \(|G_{ab}(z)|\lesssim N^{\zeta+\xi}\) for \(\Im z\sim N^{-1-\zeta}\), we also need an analogous a priori bound for \((GAG)_{ab}\) that exploits the fact that \(A\) is traceless, see~\eqref{eq:impr} later.
Consider the Ornstein-Uhlenbeck flow
\begin{equation}\label{eq:OUmain}
\dif \widehat{W}_t=-\frac{1}{2}\widehat{W}_t \dif t+\frac{\dif \widehat{B}_t}{\sqrt{N}}, \qquad \widehat{W}_0=W,
\end{equation}
with \(\widehat{B}_t\) a real symmetric Brownian motion. The OU-flow~\eqref{eq:OUmain} has the effect of adding a small Gaussian component to \(W\), so that for any fixed \(T\) we can decompose 
\begin{equation}\label{eq:imprel}
\widehat{W}_T\stackrel{\mathrm{d}}{=} \sqrt{1-cT}\widetilde{W}+\sqrt{c T}U,
\end{equation}
with \(c=c(T)>0\) a constant very close to one as long as \(T\ll 1\), and \(U,\wt W\) being independent GOE/Wigner matrices. Now let \(W_t\) be the solution of the flow~\eqref{eq:matdbm} with initial condition \(W_0=\sqrt{1-cT}\wt W\), so that 
\begin{equation}\label{eq dist GFT}
W_{cT}\stackrel{\mathrm{d}}{=}\widehat{W}_T. 
\end{equation}
\begin{lemma}\label{lem:GFT}
Let \(\widehat{W}_t\) be the solution of~\eqref{eq:OUmain}, and let \(\widehat{{\bm u}}_i(t)\) be its eigenvectors. Then for any smooth test function \(\theta\) of at most polynomial growth, and any fixed \(\epsilon\in(0,1/2)\) there exists an \(\omega=\omega(\theta,\epsilon)>0\) such that for any \(i\in [\delta N, (1-\delta)N]\) (with \(\delta>0\) from Theorem~\ref{theo:flucque}) and \(t=N^{-1+\epsilon}\) it holds that 
\begin{equation}\label{eq:GFT}
\E\theta\left(\sqrt{\frac{N}{2\braket{A^2}}}\braket{\widehat{\bm u}_i(t),A\widehat{\bm u}_i(t)}\right)=\E\theta\left(\sqrt{\frac{N}{2\braket{A^2}}}\braket{\widehat{\bm u}_i(0),A\widehat{\bm u}_i(0)}\right)+\mathcal{O}\left(N^{-\omega}\right).
\end{equation}
\end{lemma}
With \(T=N^{-1+\epsilon}\) and \(\theta(x)=x^n\) for some integer \(n\in\N\) it now follows that
\begin{equation}\label{eq:firststep}
\begin{split}
  \E\left[\sqrt{\frac{N}{2\braket{A^2}}}\braket{{\bm u}_i,A {\bm u}_i}\right]^n&=\E\left[\sqrt{\frac{N}{2\braket{A^2}}}\braket{\widehat{\bm u}_i(T),A\widehat{\bm u}_i(T)}\right]^n+\mathcal{O}\left(N^{-c}\right)\\
  &= \E\left[\sqrt{\frac{N}{2\braket{A^2}}}\braket{{\bm u}_i(cT),A{\bm u}_i(cT)}\right]^n +\mathcal{O}\left(N^{-c}\right)\\
  &= \bm1(n\,\, \mathrm{even})(n-1)!!+\mathcal{O}\left(N^{-c}\right),
\end{split}
\end{equation} 
for some small \(c=c(n,\epsilon)>0\), with \({\bm u}_i,\wh{\bm u}_i(t),\bm u_i(t)\) being the eigenvectors of \(W,\wh W_t,W_t\), respectively, concluding the proof of Theorem~\ref{theo:flucque}. Note that in~\eqref{eq:firststep} we used Lemma~\ref{lem:GFT} in the first,~\eqref{eq dist GFT} in the second and~\eqref{eq:cltgcomp} in the third step, using that in distribution the eigenvectors of \(W_{cT}\) are equal to those of \(\wt W_{cT/(1-cT)}\) with \(\wt W_t\) being the solution to the DBM flow with initial condition \(\wt W_0=\wt W\).

\begin{proof}[Proof of Lemma~\ref{lem:GFT}]
The proof of Lemma~\ref{lem:GFT} follows from comparing expectations of products of resolvents \(G(z)\) at scales slightly below the eigenvalue spacing, i.e.\ for \(\Im z\sim N^{-1-\zeta}\). Green function comparison for eigenvectors has been presented in~\cite{MR3103909} in details and has been used in~\cite{MR3606475, MR4156609, 2005.08425}. Since this is a standard argument, we only give an outline. Let \(W_t\) be the solution of~\eqref{eq:OUmain}, with \(W_0=W\), where \(W\) is a Wigner matrix satisfying Assumption~\ref{ass:entr}. Here we dropped the hat compared to the notation used in~\eqref{eq:OUmain} to make the presentation clearer, i.e.\ we use \(W_t\) instead of \(\widehat{W}_t\), \({\bm u}_i(t)\) instead of \(\widehat{\bm u}_i(t)\), etc. From now on by \(G_t=G_t(z)\) we denote the resolvent of \(W_t\). Note that along the flow~\eqref{eq:OUmain} the first two moments of \(W\) are preserved.

Due to level repulsion, as in~\cite[Lemma 5.2]{MR3034787}, to understand \(\sqrt{N}\braket{{\bm u}_i,A {\bm u}_i}\) it is sufficient to understand functions of \(\sqrt{N}\braket{\Im G(z)A}\) with \(\Im z\) slightly  below \(N^{-1}\), i.e.\ the local eigenvalue spacing. In order to prove~\eqref{eq:GFT}, as a consequence of $\braket{A^2}\ge \delta'$, it is enough to show that
\begin{equation}\label{eq:gftres}
\sup_{E\in(-2+\delta,2-\delta)}\abs*{\E\theta(\sqrt{N}\braket{\Im G_t(z)A})-\E\theta(\sqrt{N}\braket{\Im G_0(z)A})}\lesssim N^{-\omega},
\end{equation}
for \(z=E+\ii \eta\) for some \(\zeta>0,\omega>0\) and all \(\eta\ge N^{-1-\zeta}\), c.f.~\cite[Section 4]{MR4164858} and~\cite[Appendix A]{MR3606475}. Define
\begin{equation}\label{eq:R}
R_t:=\theta(\sqrt{N}\braket{\Im G_t(z)A}),
\end{equation}
then by It\^o's formula we have
\begin{equation}\label{eq:ito}
\E \frac{\dif R_t}{\dif t}=\E\left[-\frac{1}{2}\sum_\alpha w_\alpha(t)\partial_\alpha R_t+\frac{1}{2}\sum_{\alpha,\beta}\kappa_t(\alpha,\beta)\partial_\alpha\partial_\beta R_t\right],
\end{equation}
where \(\alpha,\beta\in [N]^2\) are double indices, \(w_\alpha(t)\) are the entries of \(W_t\), and \(\partial_\alpha:=\partial_{w_\alpha}\). Here 
\begin{equation}
\kappa_t(\alpha_1,\dots,\alpha_l):=\kappa(w_{\alpha_1}(t), \dots, w_{\alpha_l}(t))
\end{equation}
denotes the joint cumulant  of \(w_{\alpha_1}(t), \dots, w_{\alpha_l}(t)\), with \(l\in \mathbf{N}\). Note that by~\eqref{eq:momentass} it follows that \(|\kappa_t(\alpha_1,\dots,\alpha_l)|\lesssim N^{-l/2}\) uniformly in \(t \ge 0\). Performing a cumulant expansion in~\eqref{eq:ito} (see~\cite[Eq. (25)]{MR4221653} for more details) we are left with
\begin{equation}\label{eq:thirdhigher}
\E\frac{\dif R_t}{\dif t}=\sum_{l= 3}^R\sum_{\alpha_1,\ldots,\alpha_l}\kappa_t(\alpha_1,\ldots,\alpha_l)\E[\partial_{\alpha_1}\cdots\partial_{\alpha_l}R_t]+\Omega(R),
\end{equation}
where \(\Omega(R)\) is an error term, easily seen to be negligible as every additional derivative gains a further factor of \(N^{-1/2}\). In order to estimate~\eqref{eq:thirdhigher} we use \(|(G_t)_{ab}|\le N^{\zeta}\) and 
\begin{equation}\label{eq:impr}
\begin{split}
\big|(G_t(z_1)AG_t(z_2))_{ab}\big|&=\left|\sum_{ij}\frac{{\bm u}_i(a)
\braket{{\bm u}_i, A{\bm u}_j}{\bm u}_j(b)}{(\lambda_i-z_1)(\lambda_j-z_2)}\right| \\
&\lesssim N^{1/2+\xi}\left(\frac{1}{N}\sum_i \frac{1}{|\lambda_i-z_1|}\right)\left(\frac{1}{N}\sum_i \frac{1}{|\lambda_i-z_2|}\right)
 \lesssim N^{1/2+\xi+2\zeta},
\end{split}
\end{equation}
which holds with very high probability for \(z_1,z_2\in\set{z,\ov z}\). In~\eqref{eq:impr} we used the eigenvector delocalisation  \(\| {\bm u}_i\|_\infty\lesssim N^{-1/2+\xi}\) (see~\cite{MR2871147}, or~\cite{2007.09585} for the optimal bound) and the optimal a priori bound
\(|\braket{{\bm u}_i, A{\bm u}_j}|\lesssim N^{-1/2+\xi}\) for traceless \(A\) by~\cite[Theorem 2.2]{2012.13215} (note that this step crucially uses that \(\braket{A}=0\), the analogous bound for a general \(A\) 
would be larger  by a factor \(\sqrt{N}\)).
We claim that for any \(l\ge 0\) it holds that
\begin{equation}\label{eq:derboundgft}
|\partial_{\alpha_1}\dots \partial_{\alpha_l}\sqrt{N}\braket{\Im G_t A}|\le N^{(l+3)(\zeta+\xi)},
\end{equation}
for any arbitrary small \(\xi>0\), with very high probability. Together with
\[ \sum_{\alpha_1,\ldots,\alpha_l} \abs{\kappa_t(\alpha_1,\ldots,\alpha_l)}\lesssim N^{2-l/2} \] we are then able to estimate \(\abs{\E \dif R_t/\dif t}\) by the chain rule to finally obtain~\eqref{eq:gftres}. 

The bound~\eqref{eq:derboundgft} for \(l=0\) follows immediately from the a priori bound \(\sqrt{N}\abs{\braket{G_t A}}\lesssim N^{\xi+\zeta}\). For \(l=1\) the first derivative yields \(\abs{\partial_{ab}\sqrt{N}\braket{G_t A}}=N^{-1/2}\abs{(G_t AG_t)_{ba}}\), and each additional derivative creates a single factor of \(G_t\), and~\eqref{eq:derboundgft} follows from estimating each factor entrywise by \(\abs{(G_t)_{ab}}\lesssim N^{\xi+\zeta}\) and~\eqref{eq:impr}. 
\end{proof}

\printbibliography%

\end{document}